\documentclass[reqno]{amsart}
\usepackage[left=1in,right=1in,top=1in,bottom=1in]{geometry}
\usepackage{tikz}

\usetikzlibrary{shapes,decorations,calc,arrows}
\usepackage[foot]{amsaddr}
\usepackage{amsthm}
\usepackage{amscd}
\usepackage{amsfonts}
\usepackage{amsmath}
\usepackage{amssymb}
\usepackage{mathrsfs}
\usepackage{multirow}
\usepackage{verbatim}
\usepackage{url}
\usepackage[hidelinks]{hyperref}
\usepackage{graphicx}
\usepackage{cite}
\usepackage{fancyhdr}
\usepackage{yfonts}
\usepackage{setspace}
\usepackage{titlesec}
\usepackage{enumitem}
\usepackage{dsfont}

\titleformat{\section}[hang]
{\normalfont\Large\bfseries}
{\thesection.}{0.5em}{}

\titlespacing*{\section}{0pc}{2pc}{0.25pc}

\titleformat{\subsection}[runin]
{\normalfont\large\bfseries}
{\thesubsection}{0.5em}{}

\titlespacing{\subsection}{0pc}{1.5pc}{0.5pc}

\newcommand{\supp}{\text{supp}}
\newcommand{\Aut}{\text{Aut}}

\newcommand{\R}{\mathbb{R}}
\newcommand{\C}{\mathbb{C}}
\renewcommand{\H}{\mathcal{H}}
\newcommand{\B}{\mathcal{B}}

\newcommand{\K}{\mathcal{K}}

\renewcommand{\>}{\right\rangle}

\newcommand{\dom}{\text{dom}}

\definecolor{ggreen}{HTML}{7FDD99}

\definecolor{obnoxious}{HTML}{662288}

\newcommand{\Ad}[1]{\text{Ad}\left(#1\right)}

\newcommand{\E}{\mathcal{E}}

\newcommand{\Sd}{\operatorname{Sd}}
\renewcommand{\S}{\operatorname{S}}
\newcommand{\Sp}{\operatorname{Sp}}

\newtheorem{thm}{Theorem}[section]
\newtheorem{prop}[thm]{Proposition}
\newtheorem{lem}[thm]{Lemma}
\newtheorem{cor}[thm]{Corollary}

\theoremstyle{definition}

\newtheorem{ex}[thm]{Example}

\newtheorem{rem}[thm]{Remark}

\title{\textbf{On modular invariants of twisted group von Neumann algebras of almost unimodular groups}}
\author{Aldo Garcia Guinto$^\circ$}
\address{$\circ$Department of Mathematics, Michigan State University, East Lansing, MI 48824, USA\hfill \url{garci575@msu.edu}}
\author{Yuki Miyamoto$^\bullet$}
\address{$^\bullet$Department of Mathematics and Informatics, Chiba University, 1-33 Yayoi-cho, Inage Chiba, 263-8522, Japan \hfill \url{25wd0101@student.gs.chiba-u.jp}}
\date{}

\begin{document}

\begin{abstract}
Given a locally compact second countable group $G$ with a 2-cocycle $\omega$, we show that the restriction of the twisted Plancherel weight $\varphi^\omega_G$ to the subalgebra generated by a closed subgroup $H$ in the twisted group von Neumann algebra $L_\omega(G)$ is semifinite if and only if $H$ is open. When $G$ is almost unimodular, i.e. $\ker\Delta_G$ is open, we show that $L_\omega(G)$ can be represented as a cocycle action of the $\Delta_G(G)$ on $L_\omega(\ker\Delta_G)$ and the basic construction of the inclusion $L_\omega(\ker\Delta_G)\leq L_\omega(G)$ can be realized as a twisted group von Neumann algebra of $\Delta_G(G)\hat{\ } \times G$, where $\Delta_G$ is the modular function. Furthermore, when $G$ has a sufficiently large non-unimodular part, we give a characterization of $L_\omega(G)$ being a factor and provide a formula for the modular spectrum of $L_\omega(G)$.
\end{abstract}

\maketitle

\section*{Introduction}
Given a locally compact second countable group $G$ with a left Haar measure $\mu_G$ and Borel 2-cocycle $\omega: G \times G \to \mathbb{T}$, the twisted group von Neumann algebra $L_\omega(G)$ admits a canonical faithful normal semifinite weight $\varphi^\omega_G$, called the twisted Plancherel weight. The twisted group von Neumann algebra is generated by the left regular $\omega$-projective representation, which is defined by $(\lambda^\omega_G(s)f)(t) = \omega(t^{-1},s) f(s^{-1}t)$ for all $f \in L^2(G)$ and $s,t\in G$. For a bounded compactly supported Borel function $f :G \to \mathbb{C}$, we have
    \[
        \varphi^\omega_G(\lambda^\omega_G(f)^*\lambda^\omega_G(f)) = \| f\|_2^2,
    \]
where 
    \[
        \lambda^\omega_G(f) = \int_G f(s) \lambda_G^\omega(s) d\mu_G(s) \in L_\omega(G)
    \]
(see \cite{Sut80}). Furthermore, $L_\omega(G)$ is independent of $\omega$, up to cohomology, and there is an isomorphism of Hilbert spaces $L^2(L_\omega(G),\varphi^\omega_G)\cong L^2(G)$ which carries the modular operator $\Delta_{\varphi^\omega_G}$ of $\varphi^\omega_G$ to the modular function $\Delta_G$ of $G$ acting by pointwise multiplication. In the case, when $\omega$ is cohomologous to the trivial cocycle, this restricts to the group von Neumann algebra $L(G)$ and its Plancherel weight $\varphi_G$. 

If $G$ is discrete, the structure of $L_\omega(G)$ is better understood. Indeed, $\varphi^\omega_G$ is a faithful normal tracial state if and only if $G$ is discrete and factoriality of $L_\omega(G)$ has a characterization due to Kleppner (see \cite[Theorem 3]{Kle62}). For factoriality, the proof follows from the fact that any $x \in L_\omega(G)$ can be represented as $\sum_{s \in G} c_x(s) \lambda_G^\omega(s)$, where $c_x(s) \in \C$ and the convergence is in the strong operator topology. In general, we know that $\varphi^\omega_G$ is non-tracial if and only if $G$ is non-unimodular, but it could be the case that $L_\omega(G)$ is a semifinite von Neumann algebra. Moreover, there is no known characterization of factoriality of $L(G)$ in general. In \cite{Vaes25}, Vaes showed that there exists a locally compact group $G$ that admits a 2-cocycle such that $L(G)$ is a factor but $L_\omega(G)$ is not. This contrasts the discrete case, since $L(G)$ being a factor implies $L_\omega(G)$ is also a factor. Regardless, when $G$ is totally disconnected, Morando provides a sufficient condition for factoriality of $L(G)$ (see \cite[Theorem 1.9]{Mor25}). 

In \cite{GG25}, the first named author showed that for any 2-cocycle $\omega$, we have that $\varphi_G^\omega$ is strictly semifinite, i.e. $\varphi^\omega_G$ is semifinite on $\{ \lambda^\omega_G(s):s \in \ker\Delta_G\}''$ if and only if $G$ is almost unimodular (see \cite[Theorem 2.1]{GGN25} for no cocycle case). Here \emph{almost unimodular} means that $\ker\Delta_G$ is open (see \cite[Definition 2.2]{GGN25}). In \cite{Miy25}, the second named author showed that $\varphi_G$ is semifinite on $\{ \lambda_G(s) : s \in H\}''  \leq L(G)$ if and only if $H$ is open, and characterized factoriality of $L(G)$, when $G$ is almost unimodular with large non-unimodular part. Although the first result is known to hold for quantum groups, which is a generalization of locally compact groups, the proof is elementary in the sense that it is without any working knowledge of quantum groups. In particular, it gives an alternative proof of \cite[Theorem 2.1]{GGN25}. In this paper, we establish $\varphi_G^\omega$ is semifinite on $\{ \lambda^\omega_G(s) : s \in H\}''  \leq L_\omega(G)$ if and only if $H$ is open. Thus we generalize \cite[Theorem 2.1]{GG25} and \cite[Theorem 2.21]{Miy25}. 

In Section \ref{sec:decomp_of_twisted_gp_vNa}, we show that $L_\omega(G)$ can be represented as a twisted crossed product associated to a cocycle action of $\Delta_G(G)$ on $L_\omega(\ker\Delta_G)$, when $G$ is almost unimodular (see Theorem~\ref{thm:cocycle_action_decomp}). The existence of the cocycle action depends on $\Delta_G(G)$, which is mainly why one needs the assumption of the non-unimodular part $\Delta_G(G)$ being large. Furthermore, we establish that the point modular extension of the modular automorphism $\sigma^{\varphi^\omega_G}_t$ is the dual action associated to this cocycle action (see Corollary~\ref{cor:forms_of_the_twisted_basic_construction}). Consequently, there are many ways to represent the basic construction associated to the inclusion $L_\omega(\ker\Delta_G) \leq L_\omega(G)$, which is similar to \cite[Theorem 5.1]{GGN25} (see Corollary~\ref{cor:forms_of_the_twisted_basic_construction}). One of these representations is as a twisted group von Neumann algebra associated to $\Delta_G(G)\hat{\ } \times G$ similar to \cite[Theorem 2.2]{GG25}. We prove that factoriality of $L_\omega(\ker\Delta_G)$ tells us that all intermediate von Neumann algebras $L_\omega(\ker\Delta_G) \leq P \leq L_\omega(G)$ are of the form $P = L_\omega(H)$ for some closed intermediate group $\ker\Delta_G \leq H \leq G$ (see Theorem~\ref{thm:twisted_intermidiated_subalgs}). Lastly, we obtain $\varphi_G^\omega$ is semifinite on $\{ \lambda^\omega_G(s) : s \in H\}''  \leq L(G)$ if and only if $H$ is open (see Theorem~\ref{thm:semifinite_and_H_is_open}).

In Section \ref{sec:invariats_of_twisted_group_von_Neumann_algebras}, we analyze various invariants of the twisted group von Neumann algebras of almost unimodular groups with a 2-cocycle. First, we show that the Connes' spectrum of the cocycle action obtained by Theorem \ref{thm:cocycle_action_decomp} has a full spectrum. We therefore obtain a characterization of factoriality of the twisted group von Neumann algebras of almost unimodular groups (see Corollary \ref{cor:factoriality_of_the_twisted_version}). Next, we provide the formula of the modular spectrum of $L_\omega(G)$, similar to the group von Neumann algebra case \cite{Miy25}. That is, if we decompose $L_\omega(G) \cong L_\omega(\text{ker}\Delta_G) \rtimes_{(\check{\alpha}^\omega, u^\omega)} \Delta_G(G)$ through Theorem \ref{thm:cocycle_action_decomp}, then Proposition 3.6 tells us that:
    \begin{align*}
         \text{S}(L_\omega(G)) \backslash \{0\} &= \bigcap_{0 \ne e\in Z(L_\omega(\ker\Delta_G))^P} \overline{ \{\delta \in \Delta_G(G) : \check{\alpha}^\omega_\delta(e)e \ne 0 \}} \backslash \{0\}
         \\ & \supseteq \overline{\{ k \in \Delta_G(G) : \check{\alpha}^\omega_\delta|_{Z(L_\omega(\ker\Delta_G))} = \text{id} \}} \backslash \{0\}.
     \end{align*}
Moreover, if $\Delta_G(G)$ is discrete in $\R_+$, then the following holds:
$$\text{S}(L_\omega(G)) \backslash \{0\} = \{ \delta \in \Delta_G(G) : \check{\alpha}^\omega_\delta|_{Z(L_\omega(G_1))} = \text{id} \} \backslash \{0\}.$$
Finally, we provide a characterization of the semifiniteness of the twisted group von Neumann algebras of almost unimodular groups through the T-invariant of $L_\omega(G)$ (see Proposition \ref{prop:condition_of_semifiniteness} and \ref{prop:the_T_invariant_of_twisted_version}). This allows us to give examples of twisted group von Neumann algebras which are of type $\rm{III}_0$-factors. 

\section*{Acknowledgment} 
AGG would like to thank his advisor, Brent Nelson, for his constant support, many helpful suggestions and discussions. YM would like to thank his advisor, Hiroshi Ando, for his valuable comments. Also, YM sincerely thank Kan Kitamura at RIKEN for helpful discussion about quantum groups. AGG was supported by NSF grant DMS-2247047 and  YM was supported by JST SPRING, Grant Number JPMJSP2109. 

\section{Preliminaries}
Throughout we let $G$ denote a locally compact second countable group, which is always assumed to be Hausdorff. We will use lattice notation for the collection of closed subgroups of $G$. That is, we write $H\leq G$ to denote that $H$ is a closed subgroup of $G$, we write $H_1\vee H_2$ for the closed subgroup generated by $H_1,H_2\leq G$, etc. Similarly for a von Neumann algebra $M$ and its collection of (unital) von Neumann subalgebras and for a Hilbert space $\H$ and its collection of closed subspaces. All homomorphisms between groups are assumed to be continuous and all isomorphisms are assumed to be homeomorphisms. In particular, all representations on Hilbert spaces are assumed to be strongly continuous. A representation of a von Neumann algebra $M$ will always mean a normal unital $*$-homomorphism $\pi\colon M\to B(\H)$.

Throughout, $\mu_G$ will denote a \emph{left Haar measure} on $G$, which is a non-trivial, left translation invariant Radon measure. We follow the convention in \cite{Fol16} and take a \emph{Radon measure} to be a Borel measure that is finite on compact sets, outer regular on Borel sets, and inner regular on open sets. Such measures always exist and are unique up to scaling. Additionally, $\mu_G(U)>0$ for all non-empty open sets $U$, $\mu_G(K)<\infty$ for all compact sets $K$, and the inner regularity of $\mu_G$ extends to all $\sigma$-finite subsets (see \cite[Section 2.2]{Fol16} or \cite[Section 11]{HR79} for further details). For $1\leq p\leq \infty$ we denote by $L^p(G)$ the $L^p$-space of $G$ with respect to left Haar measures, and for a given left Haar measure $\mu_G$ we will write $\|f\|_{L^p(\mu_G)}$ for the associated $p$-norm, whereas $\|f\|_\infty$ is the unambiguous $\infty$-norm. We also denote by $\mathcal{B}(G)$ the Borel $\sigma$-algebra on $G$. Recall that the \emph{modular function} $\Delta_G\colon G\to \R_+$ is the continuous homomorphism (where $\R_+:=(0,\infty)$ has its multiplicative group structure) determined by $\mu_G(E\cdot s)= \Delta_G(s) \mu_G(E)$ for $s\in G$ and $E\in \B(G)$, where $\mu_G$ is any left Haar measure on $G$. This yields the following change of variables formulas that will be used implicitly in the sequel: 
    \[ 
        \int_Gf(st) d\mu_G(s) = \Delta_G(t)^{-1} \int_G f(s) d\mu_G(s) \qquad \text{ and } \qquad\int_G f(s^{-1}) d\mu_G(s) =  \Delta_G(s)^{-1}\int_G f(s) d\mu_G(s)
    \]
for $t\in G$ and $f\in L^1(G)$. We say $G$ is \emph{unimodular} if $\Delta_G\equiv 1$.

We assume that the reader has some familiarity with modular theory for weights on von Neumann algebras, complete details can be found in \cite[Chapter VIII]{Tak03} (see also \cite[Section 1]{GGLN25} for a quick introduction to these concepts). Given a faithful normal semifinite weight $\varphi$ on a von Neumann algebra $M$, we denote
    \begin{align*}
        \sqrt{\dom}(\varphi)&:=\{x\in M\colon \varphi(x^*x)<+\infty\}\\
        \dom(\varphi)&:=\text{span}\{x^*y\colon x,y\in \sqrt{\dom}(\varphi)\}.
    \end{align*}
We write $L^2(M,\varphi)$ for the completion of $\sqrt{\dom}(\varphi)$ with respect to the norm induced by the inner product
    \[
        \<x,y\>_\varphi:= \varphi(y^*x) \qquad x,y\in \sqrt{\dom}(\varphi).
    \]
The \emph{modular conjugation} and the \emph{modular operator} for $\varphi$ will be denoted by $J_\varphi$ and $\Delta_\varphi$, respectively. The \emph{modular automorphism group} of $\varphi$, which we view as an action $\sigma^\varphi\colon \R\curvearrowright M$, is then defined by
    \[
        \sigma_t^\varphi(x):= \Delta_\varphi^{it} x \Delta_\varphi^{-it}.
    \]
Then the \emph{centralizer} of $\varphi$ is the fixed point subalgebra under this action and is denoted
    \[
        M^\varphi :=\{x\in M\colon \sigma_t^\varphi(x) = x \ \forall t\in \R\}.
    \]
We also recall that $x\in M^\varphi$ if and only if $xy,yx\in \dom(\varphi)$ with $\varphi(xy) = \varphi(yx)$ for all $y\in \dom(\varphi)$ (see \cite[Theorem VIII.2.6]{Tak03}). That is, $M^\varphi$ is the largest von Neumann algebra of $M$ on which $\varphi$ is tracial.

Recall that a weight $\varphi$ is \emph{strictly semifinite} if its restriction to $M^\varphi$ is semifinite; that is, if $\dom(\varphi)\cap M^\varphi$ is dense in $M^\varphi$ in the strong (or weak) operator topology. This property has several equivalent characterizations (see \cite[Lemma 1.2]{GGLN25}), but the most relevant to this article is the existence of a faithful normal conditional expectation $\E_\varphi\colon M\to M^\varphi$. After \cite{Con72}, we say a faithful normal semifinite weight $\varphi$ is \emph{almost periodic} if its modular operator $\Delta_\varphi$ is diagonalizable. We will write $\Sd(\varphi)$ for the point spectrum of $\Delta_\varphi$, so that
    \[
        \Delta_\varphi = \sum_{\delta\in \Sd(\varphi)} \delta 1_{\{\delta\}}(\Delta_\varphi)
    \]
whenever $\varphi$ is almost periodic. Note that an almost periodic weight is automatically strictly semifinite by \cite[Proposition 1.1]{Con74}.

\subsection{Twisted group von Neumann algebra}\label{subsec:Twisted_gp_vNa}
In this subsection, we remind the reader of the twisted group von Neumann algebra associated to a locally compact group with a 2-cocycle and its canonical weight. 

Recall that a \emph{(Borel) 2-cocycle} $\omega: G \times G \to \mathbb{T}$ on $G$ is a Borel function satisfying
    \[
        \omega(s,t)\omega(st,r) = \omega(t,r)\omega(s,tr)  \qquad s,t,r \in G.
    \]
We say that a 2-cocycle is \emph{normalized} if $\omega(s,e)=\omega(e,s) = 1$ for all $s \in G$. Further, a 2-cocycle is \emph{fully normalized} if it is normalized and satisfies $\omega(s,s^{-1})=1$ for all $s \in G$. Note that the pointwise conjugate $\overline{\omega}$ of a fully normalized 2-cocycle $\omega$ is also fully normalized 2-cocycle and satisfies 
    \[ 
        \overline{\omega(s,t)}=\omega(t^{-1},s^{-1}).
    \]
Two 2-cocycles $\omega_1$ and $\omega_2$ are cohomologous, denoted by $\omega_1 \sim \omega_2$, if there exists a Borel map $\rho: G \to \mathbb{T}$ such that $\rho(e) =1$ and 
    \[
        \omega_1(s,t) = \overline{\rho(s)\rho(t)}\rho(st)\omega_2(s,t) \qquad s,t \in G. 
    \]
Any 2-cocycle $\omega$ is cohomologous to a fully normalized 2-cocycle (see \cite[Section 1]{Kle74}).

Let $\lambda^\omega_G,\rho^\omega_G\colon G\to B(L^2(G))$ be the \emph{left and right regular $\omega$-projective representations} of $G$:
    \[
        [\lambda^\omega_G(s)f](t)=\omega(t^{-1},s)f(s^{-1}t) \qquad \text{ and } \qquad [\rho^\omega_G(s)f](t)=\Delta_G(s)^{1/2}\omega(t,s)f(ts)
    \]
where $s,t\in G$ and $f\in L^2(G)$. The \emph{twisted group von Neumann algebra} of $G$ is the von Neumann algebra generated by its left regular $\omega$-projective representation, $L_\omega(G):=\lambda^\omega_G(G)''$. We denote the von Neumann algebra generated by its right regular $\omega$-projective representation by $R_\omega(G):=\rho^\omega_G(G)''$. By \cite[Theorem 3.8]{Sut80}, we have that $R_{\overline{\omega}}(G)=L_\omega(G)'\cap B(L^2(G))$. By \cite[Proposition 2.4]{Sut80}, any two cohomologous 2-cocycles generate the same twisted von Neumann algebra, and so we can (and will) assume that $\omega$ is a fully normalized 2-cocycle for the rest of this subsection. For any $f\in L^1_\omega(G)$,
    \begin{align}\label{eqn:L1_function_operator}
        \lambda^\omega_G(f):=\int_G  f(t)\lambda^\omega_G(t)\ d\mu_G(t)
    \end{align}
defines an element on $L_\omega(G)$ such that $\| \lambda_G^\omega(f)\| \leq \|f\|_{L^1(\mu_G)}$, where $L^1_\omega(G)$ is the same norm space as $L^1(G)$ but equipped with the twisted convolution
    \[
        (f*_\omega g)(s) = \int_G \omega(t,t^{-1}s)f(t)g(t^{-1}s)\ d\mu_G(t)=\int_G \overline{\omega(t^{-1},s)}f(t)g(t^{-1}s)\ d\mu_G(t)
    \]
and the involution
    \[
        (f^\sharp)(s) = \Delta_G(s)^{-1}\overline{f(s^{-1})}.
    \]
These operations can be found in \cite[Definition 2.7]{Sut80} or \cite[Section 2]{EL69}, the convolutions will agree with the former after a change of variables. The mapping $f\mapsto \lambda^\omega_G(f)$ gives a $*$-homomorphism from $L_\omega^1(G)$ to $L_\omega(G)$ and the $*$-subalgebra $\lambda^\omega_G(L^1_\omega(G))$ is dense in $L_\omega(G)$ in the strong (and weak) operator topology (see \cite[Theorem 2.14]{Sut80}). 

Let $B_{b,c}(G)$ be the space of equivalence classes of bounded compactly supported Borel functions, where the equivalence is given by equality $\mu_G$-almost everywhere. We endow $B_{b,c}(G)$ with the twisted convolution, involution and inner product coming from being a subset of $L^1_\omega(G) \cap L^2(G)$, which is well-defined by \cite[Lemma 2.8]{Sut80}. Then $B_{b,c}(G)$ is a left Hilbert algebra, which is dense in $L^2(G)$ and $\lambda_G^\omega(B_{b,c}(G))''= L_\omega(G)$ (see \cite[Theorem 3.3]{Sut80}). We call the elements in the full Hilbert algebra associated to $B_{b,c}(G)$ by \textit{twisted left convolvers}, since it can be shown that $f \in L^2(G)$ is in the full left Hilbert algebra if $f *_\omega g \in L^2(G)$ for each $g \in L^2(G)$. The \emph{twisted Plancherel weight} associated to $\mu_G$ is a faithful normal semifinite weight defined on $L_\omega(G)_+$ by
    \[
        \varphi^\omega_G(x^*x) := \begin{cases}
            \|f\|_{L^2(\mu_G)}^2 & \text{if }x=\lambda^\omega_G(f)\text{, with $f$ a twisted left convolver} \\
            +\infty & \text{otherwise}
        \end{cases}.
    \]
In particular, we have $L^2( L_\omega(G), \varphi_G^\omega) = L^2(G)$ and the modular objects associated to the twisted Plancherel weight are 
    \begin{equation*}
        (S_{\varphi_G^\omega}f)(s)= \Delta_G(s)^{-1} \overline{f(s^{-1})}, \qquad
        (\Delta_{\varphi_G^\omega}f)(s)=\Delta_G(s)f(s),\qquad
        (J_{\varphi_G^\omega}f)(s)=\Delta_G(s)^{-1/2}\overline{f(s^{-1})},
    \end{equation*}
where $f \in L^2(G)$ and $s \in G$ (see \cite[Lemma 3.6]{Sut80}). Thus the modular automophism group ${\sigma^{\varphi^\omega_G}_t}$ is completely determined by $\Delta_G$:
    \[  
        \sigma^{\varphi^\omega_G}_t(\lambda_G^\omega(s)) = \Delta_G(s)^{it}\lambda_G^\omega(s) \qquad s \in G.
    \]  

Notice that if $\omega$ is continuous, then one can consider instead the compactly supported continuous functions with the product coming from the twisted convolution. When the 2-cocycle is cohomologous to the trivial 2-cocycle, this collapses to the group von Neumann algebra with its Plancherel weight.

\subsection{Central extension of 2-cocycle} In this subsection we remind the reader of the central extension associated to a fully normalized Borel 2-cocycle $\omega$ of a locally compact group $G$. That is, consider the Borel cocycle semidirect product $\mathbb{T}\rtimes_{(1,\omega)} G$ associated to the Borel cocycle action $(1,\omega)\colon G \curvearrowright \mathbb{T}$, where $1: G \to \Aut(\mathbb{T})$ is the trivial map (see \cite[Section 2]{Mac58} and \cite[Section 3]{EL69}). This is the set $\mathbb{T} \times G$ provided with the following operations 
    \[
        (a,s) (b,t) = (ab\omega(s,t),st) \qquad \text{ and } \qquad (a,s)^{-1} = (\overline{a},s^{-1})
    \]
and for any normalized Haar measure $\mu_{\mathbb{T}}$ on $\mathbb{T}$ and left Haar measure $\mu_G$ on $G$, there exists a unique topology on $\mathbb{T} \rtimes_{(1,\omega)}G$ such that it becomes a locally compact group with left Haar measure $\mu_\mathbb{T} \times \mu_G$ (see \cite[Theorem 7.1]{Mac57}). In particular, the modular function in this case is 
    \begin{equation}\label{eqp:modular_function_of_central_extension}
        \Delta_{\mathbb{T}\rtimes_{(1,\omega)} G}(a,s)= \Delta_G(s) \quad (a,s)\in \mathbb{T}\rtimes_{(1,\omega)} G.
    \end{equation}
In particular, $G$ is non-unimodular if and only if  $\mathbb{T}\rtimes_{(1,\omega)} G$ is also non-unimodular. By identifying $\mathbb{T} \cong \mathbb{T} \times e$ (as locally compact group), we have that $\mathbb{T}$ is a compact normal subgroup of $\mathbb{T}\rtimes_{(1,\omega)} G$. Thus $\mathbb{T}\rtimes_{(1,\omega)} G$ is a topological group central extension of $\mathbb{T}$ by $G$, since $\mathbb{T}$ is in the center of the extension. We also have that $(\mathbb{T}\rtimes_{(1,\omega)} G)/\mathbb{T} \cong G$ (as locally compact group). If the 2-cocycle over $\mathbb{T}$ is continuous, then the topology on the cocycle semidirect product is the product topology.

We define a projection $p_G$ on $L^2(\mathbb{T}\rtimes_{(1,\omega)} G)$ by
    \begin{equation}\label{eqn:projection_onto_the_twist}
        (p_{G} f)(a,s) = \int_\mathbb{T} b[\lambda_{\mathbb{T}\rtimes_{(1,\omega)} G}(b,e)^* f](a,s) d\mu_\mathbb{T}(b) \qquad f \in L^2(\mathbb{T} \rtimes_{(1,\omega)} G), (a,s) \in \mathbb{T}\rtimes_{(1,\omega)} G.
    \end{equation}
Since $\mathbb{T}$ is in the center of $\mathbb{T}\rtimes_{(1,\omega)} G$, we have that $p_{G}$ is a central projection in $L(\mathbb{T}\rtimes_{(1,\omega)} G)$. By a change of variables, $f \in p_{G} L^2(\mathbb{T}\rtimes_{(1,\omega)} G)$ if $f(a,s) = \overline{a}f(1,s)$. We define a unitary $U_{G}:p_{G} L^2(\mathbb{T}\rtimes_{(1,\omega)} G) \to  L^2(G)$ via 
    \begin{equation}\label{eqn:partial_isometry}
        (U_{G} f)(s)=f(1,s)        
    \end{equation}
satisfying
    \begin{equation}\label{eqn:quasi-equivalence_of_twisted_and_extension}
        U_{G}  \lambda_{\mathbb{T}\rtimes_{(1,\omega)} G}(a,s) = a\lambda_G^{\omega}(s) U_{G} \qquad (a,s) \in \mathbb{T}\rtimes_{(1,\omega)} G, 
    \end{equation}
and we extent it to a partial isometry on $L^2(\mathbb{T}\rtimes_{(1,\omega)} G)$. Observe that $(U^*_{G} f)(a,s)= \overline{a}f(s)$ for $f \in L^2(G)$ and $(a,s) \in \mathbb{T} \rtimes_{(1,\omega)} G$. The following result shows that we can witness $L_\omega(G)$ as a corner of $L(\mathbb{T} \rtimes_{(1,\omega)}G)$.

\begin{prop}[{\!\!\hspace{.1 cm}\cite[Proposition 2.5]{GG25}}]\label{prop:decomp_of_central_extension}
Let $\omega:G \times G \to \mathbb{T}$ be a normalized 2-cocycle of a locally compact second countable group $G$. Then there exists a homomorphism $\Omega_{G}:L(\mathbb{T}\rtimes_{(1,\omega)} G) \to  L_{\omega}(G)$ such that 
    \[
        \Omega_{G}(\lambda_{\mathbb{T}\rtimes_{(1,\omega)} G}(a,s)) = a\lambda^{\omega}_G(s) \qquad  (a,s)\in \mathbb{T}\rtimes_{(1,\omega)} G.
    \]
Furthermore, $\Omega_{G}$ restricted to $L(\mathbb{T}\rtimes_{(1,\omega)} G)p_{G}$ is an isomorphism and $\varphi_{\mathbb{T}\rtimes_{(1,\omega)} G}\circ \Omega_{G}^{-1}$ is the twisted Plancherel weight on $L_{\omega}(G)$, where $\varphi_{\mathbb{T}\rtimes_{(1,\omega)} G}$ is the Plancherel weight on $L(\mathbb{T}\rtimes_{(1,\omega)} G)$.
\end{prop}

\section{Decomposition of twisted group von Neumann algebra}\label{sec:decomp_of_twisted_gp_vNa}
We show that the twisted group von Neumann algebra $L_\omega(G)$ of an almost unimodular group $G$ with a 2-cocycle $\omega$ can be witness as a twisted crossed product associated to a cocycle action of $\Delta_G(G)$ on $L_\omega(\ker\Delta_G)$ and the basic construction for the inclusion $L_\omega(\ker\Delta_G) \leq L_\omega(G)$ can be represented as a twisted group von Neumann algebra of the group $\Delta_G(G)\hat{\ } \times G$. Towards that end, we first remind the reader of cocycle semidirect products of locally compact groups and twisted crossed products of von Neumann algebras by locally compact groups.

Recall that for locally compact groups $H$ and $N$, a continuous cocycle action $(\alpha, c): H \curvearrowright N$ is a pair of maps $\alpha: H \to \Aut(N)$ and $c: H \times H \to N$ satisfying the relations 
    \begin{align}\label{eqn:relations_of_cocycle_LCGs}
        \alpha_s\alpha_t = \Ad{c(s,t)} \alpha_{st} \qquad \text{ and } \qquad c(s,t)c(st,r) = \alpha_s(c(t,r)) c(s,tr) \qquad s,t,r\in H,
    \end{align}
and the maps $H\times N\ni (s,x)\mapsto \alpha_s(x)\in N$ and $H\times H\ni (s,t)\mapsto c(s,t)\in N$ are continuous (where the products are equipped with the product topology). We can (and will) always normalize the $2$-cocycle so that $c(s,e)=c(e,s) =e$ for all $s \in H$. In this case, the cocycle semidirect product of this action is a locally compact group, denoted by $N \rtimes_{(\alpha,c)} H$, consisting of the set $N \times H$ equipped with the product topology and the following group operations
    \[
        (x,s)(y,t) = (x\alpha_s(y)c(s,t),st), \qquad \text{ and } \qquad (x,s)^{-1}=(\alpha_s^{-1}(x^{-1} c(s,s^{-1})^{-1}),s^{-1}),
    \]
where $(x,s), (y,t) \in N \rtimes_{(\alpha,c)} H$.  Note that a left Haar measure for this group is given by the Radon product $\mu_N \hat{\times} \mu_H$ for any left Haar measures $\mu_H$ and $\mu_N$ of $H$ and $N$, respectively. 

Recall that for a von Neumann algebra $M$ and a locally compact group $H$, a (Borel) cocycle action $(\check{\alpha},u): H \curvearrowright M$ is a pair of Borel maps $\check{\alpha}: H \to \Aut(M)$ and $u: H \times H \to M^U$ satisfying the relations
    \begin{align}\label{eqn:relations_of_cocycle_vNa_and_LCG}
        \check{\alpha}_s\check{\alpha}_t = \Ad{u(s,t)} \check{\alpha}_{st}, \qquad  u(s,t)u(st,r) = \check{\alpha}_s(u(t,r)) u(s,tr), \qquad \text{ and } \quad u(s,e) = u(e,s) =1,
    \end{align}
where $s,t,r\in H$ and $M^U$ is the set of unitaries in $M$. Here the Borel structure on $\Aut(M)$ is generated by the topology of pointwise norm convergence against $M_*$ and on $M^U$ is generated by the strong$^*$-topology on $M$. In this case when $M \subset B(\H)$ for some Hilbert space $\H$, the left twisted crossed product $M \rtimes_{(\check{\alpha},u)} H$ of $M$ by $H$ is the von Neumann algebra on $L^2(H,\H)$ generated by the operators
    \begin{align*}
        (I^{\check{\alpha}}(x) \xi)(s) &= \check{\alpha}_{s^{-1}}(x)\xi(s),\\
        (\lambda^u_H(t)\xi)(s) &= u(s^{-1},t) \xi(t^{-1}s),
    \end{align*}
where $x \in M$, $s,t \in H$, and $\xi \in L^2(H,\H)$. By \cite[Proposition 2.3]{Sut80}, the twisted crossed product is independent of the representation of $M$.
    
\begin{thm}\label{thm:cocycle_action_decomp}
Let $G$ be a second countable almost unimodular group with a 2-cocycle $\omega:G \times G \to \mathbb{T}$. Then there exists a cocycle action $(\check{\alpha}^\omega,u^\omega): \Delta_G(G) \curvearrowright L_\omega(\ker\Delta_G)$ such that $L_\omega(G) \cong L_\omega(\ker\Delta_G) \rtimes_{(\check{\alpha}^\omega,u^\omega)} \Delta_G(G)$. Furthermore, the twisted Plancherel weight on $L_\omega(G)$ is the dual weight associated to the twisted Plancherel weight on $L_\omega(\ker\Delta_G)$ .
\end{thm}
\begin{proof}
Assume that $\omega$ is fully normalized (see discussion before \cite[Lemma 2]{Kle74} and \cite[Proposition 2.4]{Sut80}). Denote $G_1:=\ker\Delta_G$, $D:=\Delta_G(G)$, $G(\omega) := \mathbb{T}\rtimes_{(1,\omega)} G$, and $G_1(\omega) : = \mathbb{T}\rtimes_{(1,\omega)} G_1$. Let $\sigma: D \to G$ be a normalized section for $\Delta_G:G \to D$; that is, $\sigma(1)=e$ and $\Delta_G \circ\sigma(\delta)=\delta$ for all $\delta \in D$. Since $\Delta_{G(\omega)}(G(\omega)) = D$ and $\ker\Delta_{G(\omega)} = G_1(\omega)$, we defined a normalized section $\overline{\sigma}: D\to G(\omega)$ for $\Delta_{G(\omega)}: G(\omega) \to D$ via $\overline{\sigma}(\delta) := (1,\sigma(\delta))$. 

Then map $((a,s),\delta) \mapsto (a,s) \overline{\sigma}(\delta)$ is an isomorphism of locally compact groups between $ G_1(\omega) \rtimes_{(\alpha,c)}D$ to $G(\omega)$, where $\alpha: D \to \Aut(G_1(\omega))$ and $c:D \times D \to G_1(\omega)$ are defined by $\alpha_\delta(a,s):=\overline{\sigma}(\delta)(a,s)\overline{\sigma}(\delta)^{-1}$ and $c(\delta_1,\delta_2) := \overline{\sigma}(\delta_1)\overline{\sigma}(\delta_2)\overline{\sigma}(\delta_1\delta_2)^{-1}$ for $\delta,\delta_1,\delta_2 \in D$ and $(a,s) \in G_1(\omega)$.

By \cite[Proposition 3.1.7]{Sut80}, there exists a cocycle action $(\check{\alpha},u)$ of $D$ on $L(G_1(\omega))$, which depends on $(\alpha,c)$, such that $L(G(\omega)) \cong L(G_1(\omega)) \rtimes_{(\check{\alpha},u)} D$. To give the explicit mapping for the case of the twisted case, we follow the proof of \cite[Proposition 4.1]{Miy25}. For each $\delta \in D$, defined a unitary $v_\delta$ on $L^2(G_1(\omega))$ by
    \[
        (v_\delta f)(a,s) := \delta^{-1/2}f(\alpha_{\delta^{-1}}(a,s)), 
    \]
where $f \in L^2(G_1(\omega))$ and $(a,s) \in G_1(\omega)$. Then the cocycle action $(\check{\alpha},u)$ of $D$ is given by 
    \begin{align*}
        \check{\alpha}_\delta(x) &:= v_\delta xv_\delta^*,\\
        u(\delta_1,\delta_2) &:= \lambda_{G_1(\omega)}(\overline{\sigma}(\delta_1^{-1})^{-1}\overline{\sigma}(\delta_2^{-1})^{-1} \overline{\sigma}(\delta_1^{-1}\delta_2^{-1})),
    \end{align*}
where $\delta, \delta_1,\delta_2 \in D$, and $(a,s) \in G_1(\omega)$. Note that $\check{\alpha}_\delta\left(\lambda_{G_1(\omega)}(a,s)\right) =\lambda_{G_1(\omega)}(\alpha_{\delta^{-1}}^{-1}(a,s))$ for $(a,s) \in G_1(\omega)$.  We identifying $L^2(G(\omega)) \cong L^2(G_1(\omega)\times D)$, and note that $\delta d\mu_{G(\omega)}(a,s)d\mu_D(\delta)$ is a left Haar measure on $G_1(\omega) \rtimes_{(\alpha,c)} D$, where $\mu_D$ is the counting measure. Define a unitary $V$ on $L^2(G_1(\omega)\times D)$ to $L^2(G_1(\omega) \rtimes_{(\alpha, c)} D)$ by
    \[
        (V\xi)((a,s),\delta)= \xi (\alpha_{\delta}^{-1}(a,s),\delta),
    \]
where $((a,s),\delta) \in G_1(\omega) \times D$. For each $(a,s) \in G_1(\omega)$ and $\delta \in D$, the proof of \cite[Proposition 4.1]{Miy25} gives us
    \begin{align*}
        V I^{\check{\alpha}}(\lambda_{G_1(\omega)}(a,s))V^* &= \lambda_{G(\omega)}( (a,s),1),\\
        V\lambda^{u}_D(\delta)V^* &= \lambda_{G(\omega)}(\overline{\sigma}(\delta^{-1})^{-1} \overline{\sigma}(\delta)^{-1},\delta).
    \end{align*}

Keeping the identification of $G(\omega)\cong G_1( \omega) \rtimes_{(\alpha, c)} D$, the homomorphism $\Omega_{G} :L(G(\omega)) \to L_\omega(G)$ of Proposition \ref{prop:decomp_of_central_extension} becomes $\Omega_{G}(\lambda_{G(\omega)}((a,s), \delta)) = \Omega_{G}(\lambda_{G(\omega)}(a\omega(s,\sigma(\delta)),s\sigma(\delta))= a\lambda_G^\omega(s)\lambda^\omega_G(\sigma(\delta))$, where  $((a,s),\delta) \in G(\omega)$. Thus we obtain a cocycle action $(\check{\alpha}^\omega, u^\omega): D \curvearrowright L_\omega(G_1)$ defined as follows 
    \begin{align*}
        \check{\alpha}^\omega_\delta &:= \Omega_{G_1} \circ \check{\alpha}_\delta \circ \Omega_{G_1}^{-1}, \\
        u^\omega(\delta_1,\delta_2) &:= \Omega_{G_1}(u(\delta_1,\delta_2)), 
    \end{align*}
where $\delta, \delta_1,\delta_2 \in D$, and $\Omega_{G_1}: L(G_1(\omega)) \to L_\omega(G_1)$ is the homomorphism of Proposition \ref{prop:decomp_of_central_extension}. Since $\check{\alpha}_\delta(\lambda_{G_1(\omega)}(a,e)) = \lambda_{G_1(\omega)}(a,e)$ and $V I^{\check{\alpha}}(\lambda_{G_1(\omega)}(a,e))V^* =\lambda_{G(\omega)}(a,e)$ for all $a \in \mathbb{T}$ and $\delta \in D$, it follows that for any $s \in G_1$ and $\delta\in D$
    \begin{align*}
        V (U_{G_1}^* \otimes 1) (I^{\check{\alpha}^\omega}(\lambda_{G_1}^\omega(s)))(U_{G_1} \otimes 1)V^* &=  V I^{\check{\alpha}}(\lambda_{G_1(\omega)}(1,s) p_{G_1})V^*=  \lambda_{G(\omega)}((1,s),1)p_{G},\\
        V(U_{G_1}^* \otimes 1)(\lambda^{u^\omega}_D(\delta))(U_{G_1} \otimes 1)V^* &=\lambda_{G(\omega)}\big(\overline{\sigma}(\delta^{-1})^{-1}\overline{\sigma}(\delta)^{-1}, \delta\big)p_G,
    \end{align*}
where $p_{G_1}$ (resp. $p_G$) and $U_{G_1}$ (resp. $U_G$) are the central projection given by equation (\ref{eqn:projection_onto_the_twist}) and the partial isometry given by equation (\ref{eqn:partial_isometry}) associated to $G_1(\omega)$ (resp. $G(\omega)$), respectively. Therefore, $\Omega_{G} \circ \Ad{V} \circ \Ad{U^* \otimes 1}$ gives us a normal, unital $*$-isomorphism from $L_\omega(G)$ to $L_\omega(G_1) \rtimes_{(\check{\alpha}^\omega,u^\omega)} D$ such that  $I^{\check{\alpha}^\omega}(\lambda_{G_1}^\omega(s)) \mapsto \lambda_G^\omega(s)$ and $\lambda^{u^\omega}_D(\delta) \mapsto  \lambda^\omega_G(\sigma(\delta^{-1})^{-1}) $ for any $s \in G_1$ and $\delta \in D$.
\end{proof}

Next, we show that the basic construction of the inclusion $L_\omega(\ker\Delta_G) \leq L_\omega(G)$ has several presentations. Similar to \cite[Theorem 5.1]{GGLN25}, one comes from the dual action of $(\alpha^\omega,u^\omega)$, which is the point modular extension of the modular automorphism group associated to $\varphi^\omega_G$ (see also \cite[Corollary 4.3]{Miy25}). Additionally, it can be presented as another twisted group von Neumann algebra on the group $\Delta_G(G)\hat{\ } \times G$, similar to \cite[Theorem 2.2]{GG25}. But first we remind the reader of dual actions and the point modular extension. 

Recall that for a cocycle action $(\check{\alpha},u)$ of a locally compact abelian group $H$ on a von Neumann algebra $M$ and $\gamma \in \widehat{H}$, we defined a unitary $\mu(\gamma)$ on $L^2(H, \mathcal{H})$ by 
    \[
        (\mu(\gamma)\xi)(s)=\overline{(\gamma \mid s)}\xi(s) \qquad \xi\in L^2(H, \mathcal{H}), s\in H,
    \]
where $(\cdot\mid\cdot) : \widehat{H} \times H \to \mathbb{T}$ will always denote dual pairings between locally compact abelian groups. Then $\hat{\alpha}_\gamma=\text{Ad}\mu(\gamma)$ defines a continuous action of $\widehat{H}$ on $M \rtimes_{(\check{\alpha},u)}H$ which satisfies 
\begin{align*}
    \hat{\alpha}_\gamma(I^{\check{\alpha}}(x)) &= I^{\check{\alpha}}(x), \\
    \hat{\alpha}_\gamma(\lambda_H^u(t)) &= \overline{ ( \gamma \mid t)}\lambda_H^u(t),
\end{align*}
where $x \in M, t \in H$ and $\gamma \in \widehat{H}$ (see \cite[II, P150]{Sut80}). This action is called the \textit{dual action} of $(\check{\alpha},u)$. 

When $G$ is almost unimodular, and equivalently $\varphi^\omega_G$ is almost periodic, the modular automorphism group $\sigma^{\varphi^\omega_G}\colon \R\curvearrowright L_\omega(G)$ admits an extension, called the \emph{point modular extension} of $\sigma^{\varphi^\omega_G}$, $\widetilde{\sigma}\colon \Delta_G(G)\hat{\ } \curvearrowright L_\omega(G)$ satisfying $\widetilde{\sigma}_{\hat{\iota}(t)} = \sigma_t^{\varphi^\omega_G}$ for all $t\in \R$, where $\hat{\iota}\colon \R\to \Delta_G(G)\hat{\ }$ is the map dual to the inclusion map $\iota\colon \Delta_G(G) \hookrightarrow \R_+$; that is,
    \begin{equation}\label{eqn:transpose_formula}
        (\hat{\iota}(t)\mid \delta) = ( t \mid  \iota(\delta)) = \delta^{it} \qquad t\in \R,\ \delta\in \Delta_G(G)
    \end{equation}
(see \cite[Section 1.4]{GGLN25} for more details).

\begin{cor}\label{cor:forms_of_the_twisted_basic_construction}
Let $G$ be a second countable almost unimodular group with a 2-cocycle $\omega:G \times G \to \mathbb{T}$, let $\varphi^\omega_G$ be a twisted Plancherel weight on $L_\omega(G)$ and let $\widetilde{\sigma}\colon \Delta_G(G)\hat{\ }\curvearrowright L_\omega(G)$ be the point modular extension of $\sigma^{\varphi^\omega_G}\colon \R\curvearrowright L_\omega(G)$. Then one has 
\[
    \langle L_\omega(G), e_{\varphi^\omega_G}\rangle \cong L_\omega(G) \rtimes_{\widetilde{\sigma}}\Delta_G(G)\hat{\ } \cong L_\omega(\ker\Delta_G) \bar\otimes B(\ell^2\Delta_G(G)) \cong L_{\widetilde{\omega}}(\Delta_G(G)\hat{\ }\times G),
\]
where $\widetilde{\omega}$ is a 2-cocycle on $\Delta_G(G)\hat{\ } \times G$ defined by
    \[
        \widetilde{\omega}\big( (\gamma_1,s_1)(\gamma_2,s_2)\big) := \overline{(\gamma_2\mid\Delta_G(s_1))}\omega(s_1,s_2)
    \]
where $(\cdot\mid \cdot)\colon \Delta_G(G)\hat{\ }\times \Delta_G(G)\to \mathbb{T}$ is the dual pairing and $(\gamma_1,s_1), (\gamma_2,s_2) \in \Delta_G(G)\hat{\ }\times G$.
\end{cor}
\begin{proof} We denote $G_1:= \ker\Delta_G$ and $D:=\Delta_G(G)$. For each $\gamma \in \widehat{D}$, we defined 
    \[
        U_\gamma := \sum_{\delta \in D} (\gamma|\delta) 1_{\Delta_G^{-1}(\{\delta\})}.
    \]
By computation, we have $\widetilde{\sigma}_\gamma(x) = U_\gamma x U_\gamma^*$ for all $\gamma \in \widehat{D}$ and $x \in L_\omega(G)$. Hence $\langle L_\omega(G), e_{\varphi^\omega_G}\rangle = L_\omega(G) \vee \{U_\gamma: \gamma \in \widehat{D}\} ''$. Next, we define an action $\beta: \widehat{D} \curvearrowright R_{\overline{\omega}}(G)$ via $\beta_\gamma(y) := U_\gamma y U_\gamma^*$, which  satisfies $\beta_\gamma(\rho_G^{\overline{\omega}}(s))= \overline{(\gamma |\Delta_G(s))}\rho_G^{\overline{\omega}}(s)$ for all $\gamma \in \widehat{D}$ and $s \in G$. Consider the family of projections $\{ p_\delta\}_{\delta \in D}$ on the crossed product $R_{\overline{\omega}}(G) \rtimes_\beta \widehat{D}$ defined by
    \[
        p_\delta := \int_{\widehat{D}} (\gamma | \delta) \lambda_{\widehat{D}}(\gamma) d\mu_{\widehat{D}}(\gamma) \qquad \delta \in D,
    \]
where $\mu_{\widehat{D}}$ is the unique Haar measure on $\widehat{D}$ satisfying $\mu_{\widehat{D}}(\widehat{D}) =1$. By \cite[Theorem 2.2]{Hag76}, the first isomorphism follows by showing that $p_1$ has full central support in $R_{\overline{\omega}}(G) \rtimes_\beta \widehat{D}$. Rewriting the covariant relation, we have $\rho^{\overline{\omega}}_G(s)\lambda_{\widehat{D}}(\gamma)\rho^{\overline{\omega}}_G(s)^*= (\gamma | \Delta_G(s))\lambda_{\widehat{D}}(\gamma)$ and subsequently, 
    \[
        \rho^{\overline{\omega}}_G(s)p_1\rho^{\overline{\omega}}_G(s)^* = \int_{\widehat{D}} (\gamma|\Delta_G(s)) \lambda_{\widehat{D}}(\gamma) d\mu_{\widehat{D}}(\gamma) = p_{\Delta_G(s)}.
    \]
Thus the central support of $p_1$ is 1 since $\sum_\delta p_\delta=1$.
    
The second isomorphism follows by showing that $\widetilde{\sigma}$ is the dual action of the cocycle action $(\check{\alpha}^\omega,u^\omega)$ on $L_\omega(G_1)$ from Theorem \ref{thm:cocycle_action_decomp}, and applying \cite[Theorem 2]{NS79}. We identify $G \cong G_1 \rtimes_{(\alpha, c)} D$ associated to a normalized section $\sigma$ and let $\Phi:L_\omega(G) \to L_\omega(G_1) \rtimes_{(\check{\alpha}^\omega, u^\omega)} D$ be the isomorphism of Theorem~\ref{thm:cocycle_action_decomp}. Note that, under the identification of $G$, we have $\Delta_G(s,\delta) = \delta$  for all $(s,\delta) \in G$. For each $t \in \R$ and $(s,\delta)\in G$,
    \[
        \widetilde{\sigma}_{\hat{\iota}(t)}(\lambda^\omega_G(s,\delta))= \sigma_t^{\varphi^\omega_G}(\lambda^\omega_G(s,\delta))= \Delta_G(s,\delta)^{it} \lambda^\omega_G(s,\delta)=\delta^{it}\lambda^\omega_G(s,\delta)=(\iota(\delta)\mid t) \lambda_G^\omega(s,\delta)= (\delta\mid \hat{\iota}(t)) \lambda_G^\omega(s,\delta).
    \]
The density of $\hat{\iota}(\R) \subset D$ and continuity $\widehat{D} \ni \gamma \mapsto \widetilde{\sigma}_\gamma \in \Aut(L_\omega(G))$ gives us that 
    \[
        \widetilde{\sigma}_\gamma(\lambda^\omega_G(s,\delta)) = (\delta\mid \gamma) \lambda^\omega_G(s,\delta) \qquad \gamma \in \widehat{D}, \delta \in D, s \in G_1.
    \]
It follows that $\widetilde{\sigma}$ is the dual action of $(\check{\alpha}^\omega,u^\omega): D\curvearrowright L_\omega(G_1)$. Thus by \cite[Theorem 2]{NS79}, we have
    \[
        L_\omega(G) \rtimes_{\widetilde{\sigma}} \widehat{D} \cong [L_\omega(G_1) \rtimes_{(\check{\alpha}^\omega,u^\omega)}D]\rtimes_{\widetilde{\sigma}} \widehat{D} \cong L_\omega(G_1) \bar\otimes B(\ell^2D)
    \]
as claimed.

To show the last isomorphism, we show that $L_\omega(G) \rtimes_{\widetilde{\sigma}} \widehat{D}\cong L_{\widetilde{\omega}}(\widehat{D} \times G)$. By computing the cocycle identity, we have that $\widetilde{\omega}$ is indeed a 2-cocycle. Recall that $L_\omega(G)\rtimes_{\widetilde{\sigma}} \widehat{D}$ is generated by the operators $(\lambda_{\widehat{D}}(\gamma)f)(\gamma_1,t) = f(\gamma^{-1}\gamma_1,t)$ and $[\pi_{\widetilde{\sigma}}(\lambda^\omega_G(s))f](\gamma_1,t) = \overline{(\gamma_1\mid \Delta_G(s))}\omega(t^{-1},s)f(\gamma_1,s^{-1}t)$, where $(\gamma_1, t) \in \widehat{D} \times G$, $s \in G$, $\gamma \in \widehat{D}$, and $f \in L^2(\widehat{D} \times G)$. Thus the map that sends $\lambda_{\widehat{D}\times G}^{\widetilde{\omega}}(\gamma,s)$ to $\lambda_{\widehat{D}}(\gamma)\pi_{\widetilde{\sigma}}(\lambda^\omega_G(s))$ defines the desired normal unital $*$-isomorphism.
\end{proof}

\subsection{Analysis of the twisted Plancherel weight}
In this subsection, we show that for any closed subgroup $H$ of $G$, $H$ is open if and only if the twisted Plancherel weight $\varphi_G^\omega$ is semifinite on the subalgebra $\{ \lambda^\omega_G(s): s \in H\}'' \leq L_\omega(G)$, where $\omega$ is a 2-cocycle on $G$. This is the projective analogue of \cite[Theorem 2.21]{Miy25}. Towards that end, we remind the reader of rho-functions and the support of an element of the group von Neumann algebra.

Recall that for a closed subgroup $H$ of a locally compact group $G$, there exists a continuous function $\rho\colon G\to (0,\infty)$, called \emph{rho-function} for the pair $(G,H)$, satisfying
    \[
        \rho(st) = \frac{\Delta_H(t)}{\Delta_G(t)} \rho(s), \qquad \qquad s\in G,\ t\in H,
    \]
and there exists a strongly quasi-invariant measure $\mu_{G/H}$ on $G/H$ such that 
    \[
        \int_G f(s) \rho(s) d\mu_G(s) = \int_{G /H} \int_H f(st) d\mu_H(t) d\mu_{G/H}(sH) \qquad f \in C_c(G)
    \]
(see \cite[Proposition 5.4 and Theorem 2.56]{Fol16}). 

Recall that the \emph{Fourier algebra} $A(G)$ of $G$ is the set of all $f * g^\flat$ such that $f,g \in L^2(G)$, where $g^\flat(s)=\overline{g(s^{-1})}$. We have that $A(G)$ is isomorphic to the predual $L(G)_*$ of $L(G)$ as Banach spaces. We defined the duality pairing by 
    \[
        \langle x, \xi\rangle = \langle xf,g\rangle_{L^2(\mu_G)} \qquad \xi :=f*g^\flat \in A(G), x \in L(G).
    \]
For $\xi \in A(G)$, we defined a normal completely bounded map $m_\xi: L(G) \to L(G)$ by
    \[
        \langle m_\xi(x), \eta\rangle = \langle x, \xi\eta\rangle \qquad \xi,\eta \in A(G), x \in L(G).
    \]
It is not difficult to show that $\langle\lambda_G(s), \xi\rangle = \xi(s)$ and $m_\xi(\lambda_G(s)) = \xi(s)\lambda_G(s)$. We defined the \emph{support} of $x \in L(G)$ as follows
    \[
        \supp(x):= \{ s \in G: m_\xi(x) \neq 0 \text{ for all } \xi \in A(G) \text{ such that } \xi(s) \neq 0\}.
    \]
The support of $x$ is known to be a closed set in $G$ (for more details see \cite[Section 2.6 and 3.1]{BB18}).
\begin{thm}\label{thm:semifinite_and_H_is_open}
Let $H$ be a closed subgroup of a second countable locally compact group $G$, let $\omega: G \times G \to \mathbb{T}$ be a 2-cocycle and let $\varphi^\omega_G$ be the twisted Plancherel weight on $L_\omega(G)$. Then $H$ is open if and only if $\varphi^\omega_G$ is semifinite on $\{\lambda^\omega_G(h): h \in H\}''$. In this case, under the identification of  $\{\lambda^\omega_G(h): h \in H\}'' \cong L_\omega(H)$, $\varphi_G^\omega$ is a multiple of the twisted Plancherel weight $\varphi_H^\omega$ on $L_\omega(H)$.
\end{thm}
\begin{proof}
Assume that $\omega$ is fully normalized (see discussion before \cite[Lemma 2]{Kle74} and \cite[Proposition 2.4]{Sut80}). Denote $G(\omega) := \mathbb{T}\rtimes_{(1,\omega)} G$ and $H(\omega):= \mathbb{T}\rtimes_{(1,\omega)} H$. By \cite[Proposition 2.6]{GG25}, there exists normal unital $*$-isomorphism $\Phi:L_\omega(H) \to L_\omega(G)$ such that $\Phi(\lambda^\omega_H(h)) = \lambda^\omega_G(h)$ for all $h \in H$.

Suppose $H$ is open, then $G/H$ is discrete and $\Delta_G|_H = \Delta_H$. It follows that rho-function for the pair $(G,H)$ is 1 and the counting measure is a $G$-invariant Radon measure on $G/H$. Choose left Haar measures $\mu_G$ and $\mu_H$ on $G$ and $H$, respectively, such that 
    \[
        \int_G f(s)d\mu_G(s) = \sum_{sH \in G/H} \int_H f(sh) d\mu_H, \qquad f \in C_c(G). 
    \]
The openness of $H$ and the normalization implies that for each $f \in L^p(H)$, we have a measurable function $\tilde{f}$ defined by $\tilde{f}(s):= f(s)$ if $s \in H$ and zero otherwise and satisfies $\tilde{f} \in L^p(G)_H := \{ g \in L^p(G): \supp (g) \subset H\}$ and $\|\tilde{f}\|_{L^p(\mu_G)} = \|f\|_{L^p(\mu_H)}$. Thus for each $f \in L^1_\omega(H)$, we have $\lambda^\omega_G(\tilde{f}) \in \{\lambda^\omega_G(h): h \in H\}''$ and for each $f \in L^2(H)$, we have $\varphi^\omega_G(\lambda_G^\omega(\tilde{f})^*\lambda_G^\omega(\tilde{f})) =\|f\|_{L^2(\mu_G)}^2$. It follows that $\varphi^\omega_G$ is semifinite in $\{\lambda^\omega_G(h): h \in H\}''$, since $L_\omega(H)$ has a bounded approximate identity (see \cite[Lemma 3.1]{Edw69}).

Suppose that $\varphi_G^\omega$ is semifinite on $\Phi(L_\omega(H))$. By composing it with the homomorphism $\Omega_G$ of Proposition~\ref{prop:decomp_of_central_extension}, it follows that $\varphi_{G(\omega)}$ is semifinite on $\Omega_G(\Phi(L_\omega(H)))$. By \cite{Haa79}, there exists a continuous left bounded function $f\in L^2(G(\omega))$ such that $\pi_\ell(f) \in \Omega_G(\Phi(L_\omega(H)))\cap \dom(\varphi_{G(\omega)})$, where $\pi_\ell(f)$ is the operator defined by $\pi_\ell(f)g= f *g$ for all $g \in L^2(G(\omega))$. By \cite[Lemma 2.13]{Miy25}, $\supp(f) \subset \supp(\pi_\ell(f))$. Additionally, we have $\Omega_G(\Phi(L_\omega(H))) \subset \{ \lambda_{(G(\omega)}(a,s) : (a,s) \in H(\omega)\}''$ and by \cite[Theorem 3]{TT72}, it follows that $\supp(f) \subset \supp(\pi_\ell(f)) \subset H(\omega)$. Thus $H(\omega)$ is open in $G(\omega)$ by continuity of $f$ and applying \cite[Lemma 1.1]{GGN25}. Moreover $H$ is open.

Lastly, we show that the faithful normal semifinite weight $\psi: =\varphi^\omega_G \circ \Phi$ is a multiple of $\varphi^\omega_H$. Since we normalized the measures on $G$ and $H$ such that $G/H$ has the counting measure, the multiple will be 1. Otherwise it will be $\mu(H)$, where $\mu$ is a $G$-invariant measure on $G/H$. Since $\Delta_G|_H=\Delta_H$, it follows that 
    \[
        \sigma^\psi_t(\lambda^\omega_H(h)) = \Phi^{-1} \circ \sigma^{\varphi^\omega_G}(\lambda^\omega_G(h)) = \Delta_G(h)^{it} \lambda^\omega_H(h)= \Delta_H(h)^{it} \lambda^\omega_H(h) = \sigma^{\varphi^\omega_H}(\lambda^\omega_H(h)).
    \]
Thus $\psi \circ \sigma^{\varphi_H^\omega} = \psi$, that is $\psi$ commutes with $\varphi_H^\omega$. By \cite[Proposition VIII.3.15]{Tak03}, $\psi$ and $\varphi^\omega_H$ agree if they agree on a $\sigma$-weakly dense $*$-subalgebra in $\sqrt{\dom}(\varphi^{\omega}_H)$, which is invariant over the automorphism group $\sigma^{\varphi^{\omega}_H}$. We consider the $\sigma$-weakly dense $*$-subalgebra $B = \text{span}\{ \lambda^{\omega}_H(f_1)^*\lambda_H^{\omega}(f_2) : f_1,f_2 \in B_{b,c}(H)\}$. For each $f \in B_{b,c}(H)$,  $\Phi(\lambda^\omega_H(f))= \lambda^\omega_G(F)$ is an element in $\lambda^\omega_G(B_{b,c}(G))$, where $F(s)=f(s)$ if $s \in H$ and zero otherwise. Thus by the normalization on $\mu_G$ and $\mu_H$, we have
    \[
        \psi(\lambda_H^\omega(f_1)^*\lambda^\omega_H(f_2)) = \varphi^\omega_G( \lambda^\omega_G(F_1)^*\lambda^\omega_G(F_2))= \langle F_1,F_2\rangle_{L^2(\mu_G)} = \langle f_1,f_2\rangle_{L^2(\mu_H)}=\varphi^\omega_H(\lambda_H^\omega(f_1)^*\lambda^\omega_H(f_2)),
    \]
for all $f_1,f_2 \in B_{b,c}(H)$, as desired.
\end{proof}

The following results follows from the previous result, $\{ \lambda^\omega_G(s) : s\in H\}''$ being globally $\sigma^{\varphi^\omega_G}$-invariant and \cite[Theorem IX.4.2]{Tak03}.

\begin{cor}
Let $H$ be a closed subgroup of a second countable locally compact group $G$, let $\omega: G \times G \to \mathbb{T}$ be a 2-cocycle and let $\varphi^\omega_G$ be the twisted Plancherel weight on $L_\omega(G)$. Then $H$ is open if and only if there exists a $\varphi^\omega_G$-preserving conditional expectation from $L_\omega(G)$ onto $\{ \lambda^\omega_G(s) : s\in H\}'' \cong L_\omega(H)$.
\end{cor}

\subsection{Intermediate twisted group von Neumann algebra}
In \cite[Subsection 5.1]{GGN25}, we showed that all the intermediate von Neumann algebras $L(\ker\Delta_G) \leq P \leq L(G)$ are isomorphic to $L(H)$ for some intermediate group $\ker \Delta_G \leq H \leq G$ if and only if $L(\ker \Delta_G)$ is a factor (see \cite[Theorem 5.6]{GGN25}). This followed by \cite{ILP98} and witnessing $L(H)$ as the fixed point subalgebra of the point modular extension of the modular automorphism group of $\varphi_G$. In light of Theorem~\ref{thm:cocycle_action_decomp} and Corollary~\ref{cor:forms_of_the_twisted_basic_construction}, we show that this still holds for the twisted group von Neumann algebras. Since $L(\mathbb{T}\rtimes_{(1,\omega)} \ker\Delta_G)$ is not a factor, we do not appeal to Proposition~\ref{prop:decomp_of_central_extension}. Regardless, the proof of this result is analogues to \cite[Theorem 5.6]{GGN25} since that result depended on the generators of $L(G)$. 

Each intermediate group $\ker \Delta_G \leq H \leq G$ can be identified with 
    \[
        H \cong \ker\Delta_G \rtimes_{(\beta , c)} \Delta_G(H),
    \]
where $(\beta, c):\Delta_G(G) \curvearrowright \ker\Delta_G$ is a continuous cocycle action (see \cite[Remark 3.7]{GGN25}). Then if $(\check{\alpha}^\omega, u^\omega): \Delta_G(G) \curvearrowright L_\omega(\ker\Delta_G)$ is the cocycle action associated to $(\beta,c)$ as we saw in Theorem~\ref{thm:cocycle_action_decomp}, we have 
    \[
        L_\omega(H) \cong L_\omega(\ker\Delta_G) \rtimes_{(\check{\alpha}^\omega, u^\omega)} \Delta_G(H).
    \]
Subsequently, we obtain
    \[
        L_\omega(G)^{\Delta_G(H)^\perp}:= \{ x \in L_\omega(G) : \widetilde{\sigma}_\gamma(x) =x \ \forall \gamma \in \Delta_G(H)^\perp\} \cong L_\omega(H), 
    \]
where $\Delta_G(H)^\perp := \{ \gamma \in \Delta_G(G)\hat{\ } : ( \gamma \mid \delta) = 1 \ \forall \delta \in \Delta_G(H)\}$ and $\widetilde{\sigma}: \Delta_G(G)\hat{\ } \curvearrowright L_\omega(G)$ is the point modular extension of the modular automorphism group of $\varphi^\omega_G$, which is also the dual action of $(\check{\alpha}^\omega,u^\omega)$. Conversely, for a closed subgroup $K \leq \Delta_G(G)\hat{\ }$, we denote $K_\perp:=\{ \delta \in \Delta_G(G) : (\gamma \mid \delta) = 1 \ \forall \gamma \in K\}$ and thus
    \begin{equation}\label{eqn:fixed_pnt_decomposition}
        L_\omega(G)^K \cong L_\omega( \ker\Delta_G \rtimes_{(\beta, c)} K_\perp).
    \end{equation}
It follows that an intermediate algebra is of the form $L_\omega(H)$ for some closed subgroup $H\leq G$ if and only if it is of the form $L_\omega(G)^K$ for some closed subgroup $K\leq \Delta_G(G)\hat{\ }$. The following result gives a characterization of the latter, which is the twisted version of \cite[Proposition 5.5]{GGN25}.

\begin{prop}
Let $G$ be a second countable almost unimodular group with a 2-cocycle $\omega:G \times G \to \mathbb{T}$, let $\widetilde{\sigma}: \Delta_G(G)\hat{\ } \curvearrowright L_\omega(G)$ be the point modular extension of the modular automorphism group of the twisted Plancherel weight $\varphi^\omega_G$ on $L_\omega(G)$. For an $\tilde{\sigma}$-invariant intermediate von Neumann algebra $L_\omega(\ker\Delta_G)\leq P \leq L_\omega(G)$, denote for each $\delta \in \Delta_G(G)$
    \[
        z_\delta : = \bigvee \{ vv^* : v \in P \text{ partial isometry with } \widetilde{\sigma}_\gamma(v) = (\gamma \mid \delta ) v \ \forall \gamma \in \Delta_G(G)\hat{\ }\}.
    \]
Then one has $P = L_\omega(G)^K$ for some closed subgroup $K \leq \Delta_G(G)\hat{\ }$ if and only if $z_\delta = 1_{\Sd(\varphi^\omega_G|_P)}(\delta)$ for all $\delta \in \Delta_G(G)$. In this case, $\Sd(\varphi^\omega_G|_P)$ is a group and one has $K = \Sd(\varphi^\omega_G|_P)^\perp$.
\end{prop}
\begin{proof}
Since $\varphi^\omega_G$ is semifinite on $L_\omega(\ker\Delta_G)$, it follows that it is also semifinite on $P$. Additionally, $L^2(P,\varphi^\omega_G) \leq L^2(L_\omega(G), \varphi^\omega_G)$ is an invariant subspace for $\Delta_{\varphi_G^\omega}$, since $P$ is $\widetilde{\sigma}$-invariant and $\widetilde{\sigma}$ extends $\sigma^{\varphi_G^\omega}$. Thus $\varphi_G^\omega|_P$ is almost periodic and by \cite[Lemma 2.1]{GGLN25}, $z_\delta=0$ for $\delta\notin \Sd(\varphi^\omega_G|_P)$ and for $\delta \in \Sd(\varphi^\omega_G|_P)$, $z_\delta$ in the center of $L(\ker\Delta_G)$.

Now, suppose that $P=L_\omega(G)^K$ for some closed subgroup $K \leq \Delta_G(G)\hat{\ }$. Then $\Sd(\varphi^\omega_G|_P)=K_\perp$ by (\ref{eqn:fixed_pnt_decomposition}) and  by the first part of the proof, $z_\delta=0$ if $\delta \notin K_\perp$. For $\delta \in \K_\perp$, take $s \in \Delta^{-1}_G(\{\delta\})$ so that $\lambda_G^\omega(s) \in L_\omega(G)^K =P$. Thus $z_\delta \geq \lambda_G^\omega(s)\lambda_G^\omega(s)^* = 1$. Moreover, in this case $K = (K_\perp)^\perp =  \Sd(\varphi^\omega_G|_P)^\perp$ as claimed.

Conversely, suppose $z_\delta =1_{ \Sd(\varphi^\omega_G|_P)}(\delta)$ for all $\delta \in \Delta_G(G)$. For each $\delta \in  \Sd(\varphi^\omega_G|_P)$, by \cite[Lemma 2.1]{GGLN25} there exists a family of partial isometries $\mathcal{V}_\delta \subset P$ satisfying $\sigma^{\varphi^\omega_G}_t(v) = \delta^{it}v$ for all $t \in \R$ and $v \in \mathcal{V}_\delta$ such that 
    \[
        \sum_{v \in \mathcal{V}_\delta} vv^* = z_\delta = 1.
    \]
Then for any $s \in \Delta_G^{-1}(\{\delta\})$, one has $\lambda_G^\omega(s)^*v \in L_\omega(G)^{\varphi^\omega_G}= L_\omega(\ker\Delta_G)$ for all $v \in \mathcal{V}_\delta$, and consequently, 
    \[
        \lambda_G^\omega(s)^* = \sum_{v \in \mathcal{V}_\delta} (\lambda_G^\omega(s)^*v)v^* \subset L_\omega(G)^{\varphi^\omega_G}P = P.
    \]
Thus $\lambda_G^\omega(\Delta_G^{-1}(\Sd(\varphi^\omega_G|_P)) \subset P$. It follows that $\Sd(\varphi^\omega_G|_P)$ is a group. Indeed for $\delta_1,\delta_2 \in \Sd(\varphi^\omega_G|_P)$, choose $s_i \in \Delta_G^{-1}(\{\delta_i\})$ for $i=1,2$, we obtain $\lambda_G^\omega(s_1s_2) = \overline{\omega(s_1,s_2)}\lambda_G^\omega(s_1)\lambda^\omega_G(s_2) \in P \setminus \{0\}$ and $\lambda_G(s_1)^* \in P \setminus \{0\}$ implying $\delta_1\delta_2, \delta_1^{-1} \in \Sd(\varphi^\omega_G|_P)$. Hence we consider $K = \Sd(\varphi^\omega_G|_P)^\perp$ and (\ref{eqn:fixed_pnt_decomposition}) implies $L_\omega(G)^K \leq P$. By \cite[Lemma 1.4]{GGLN25}, $P$ is generated by eigenoperators of $\sigma^{\varphi^\omega_G}$ with eigenvalues in $\Sd(\varphi^\omega_G|_P)$. Since $\widetilde{\sigma}$ extends $\sigma^{\varphi^\omega_G}$, for any eigenoperator $x \in P$ one has $\widetilde{\sigma}_\gamma(x) = (\gamma \mid \delta) x$ for all $\gamma \in \Delta_G(G)\hat{\ }$ and for some $\delta \in \Sd(\varphi^\omega_G|_P)$. Thus $P = L_\omega(G)^K$ since $x \in L_\omega(G)^K$.
\end{proof}

The following result is the twisted version of \cite[Theorem 5.6]{GGN25} and its proof, which we will omit, is similar provided one uses the left regular $\omega$-projective representation $\lambda_G^\omega(s)$ for all $s \in G$. 

\begin{thm}\label{thm:twisted_intermidiated_subalgs}
Let $G$ be a second countable almost unimodular group with a 2-cocycle $\omega:G \times G \to \mathbb{T}$ and let $\varphi^\omega_G$ be a twisted Plancherel weight on $L_\omega(G)$. Assume that $G$ is non-unimodular. Then every intermediate von Neumann algebra $L_\omega(\ker\Delta_G)\leq P \leq L_\omega(G)$ is of the form $P =L_\omega(H)$ for a closed intermediate group $\ker\Delta_G \leq H \leq G$ if and only if $L_\omega(\ker\Delta_G)$ is a factor.
\end{thm}

\section{Invariants of twisted group von Neumann algebra}\label{sec:invariats_of_twisted_group_von_Neumann_algebras}
To give the characterization of factoriality of $L_\omega(G)$, we remind the reader of the Connes' spectrum $\Gamma(\check{\alpha},u)$ of a cocycle action $(\check{\alpha},u)$ of a locally compact abelian group $H$ on a von Neumann algebra $M$. That is, $\Gamma(\check{\alpha},u)$ is defined as the kernel of the restriction of the dual action $\hat{\alpha}$ of $\widehat{H}$ to $(\check{\alpha},u)$ to the center $Z(M\rtimes_{(\check{\alpha},u)} H)$ of the twisted crossed product $M\rtimes_{(\check{\alpha},u)} H$. That is,
    \[
        \Gamma(\check{\alpha},u):=\{\gamma\in \widehat{H}: \hat{\alpha}_\gamma(z)=z \ \forall z\in Z(M\rtimes_{(\check{\alpha},u)} H)\}.
    \] 

\begin{prop}
    Let $G$ be a second countable almost unimodular group with a 2-cocycle $\omega:G \times G \to \mathbb{T}$. If we identify $L_\omega(G) \cong L_\omega(\ker\Delta_G) \rtimes_{(\check{\alpha}^\omega,u^\omega)} \Delta_G(G)$ via Theorem~\ref{thm:cocycle_action_decomp}, then we have $\Gamma(\check{\alpha}^\omega, u^\omega) = \Delta_G(G)\hat{\ }$.
\end{prop}
\begin{proof}
    Assume that $\omega$ is fully normalized (see discussion before \cite[Lemma 2]{Kle74} and \cite[Proposition 2.4]{Sut80}). Denote $G_1:=\ker\Delta_G$, $D:=\Delta_G(G)$, $G(\omega) := \mathbb{T}\rtimes_{(1,\omega)} G$, and $G_1(\omega) : = \mathbb{T}\rtimes_{(1,\omega)} G_1$. Let $(\check{\alpha},u): D \curvearrowright L(G_1(\omega))$ be an cocycle action which provides the decomposition $L(G(\omega)) \cong L(G_1(\omega)) \rtimes_{(\check{\alpha}, u)} D$. Since $\check{\alpha}_\delta(p_{G_1})=p_{G_1}$ for all $\delta \in D$, $(\check{\alpha},u): D \curvearrowright L(G_1(\omega))$ can be restricted to the cocycle action $(\check{\beta},v): D \curvearrowright L(G_1(\omega))p_{G_1}$. That is, $(\check{\beta}, v)$ is the cocycle action of $D$ on $L(G_1(\omega))p_{G_1}$ defined by
    \begin{align*}
        \check{\beta}_\delta(x) &:= \check{\alpha}_\delta(x), \\
        v(\delta_1, \delta_2) &:= u(\delta_1, \delta_2)p_{G_1}, 
    \end{align*}
    where $\delta, \delta_1, \delta_2 \in D$ and $x \in L(G_1(\omega))p_{G_1}$. Theorem~\ref{thm:cocycle_action_decomp} shows that $L(G_1(\omega))p_{G_1}\rtimes_{(\check{\beta},v)} D \cong L_\omega(G_1) \rtimes_{(\check{\alpha}^\omega, u^\omega)} D $ through the isomorphism $\Omega_{G_1}$ of Proposition~\ref{prop:decomp_of_central_extension}. Therefore, it follows that $\Gamma(\check{\alpha}^\omega, u^\omega) = \Gamma(\check{\beta}, v)$. Note that $I^{\check{\alpha}}(p_{G_1})$ is a central projection in $L(G_1(\omega)) \rtimes_{(\check{\alpha}, u)}D$, on the Hilbert space $I^{\check{\alpha}}(p_{G_1})L^2(D, L^2(G_1(\omega))) = L^2(D, p_{G_1}L^2(G_1(\omega)))$, it can be shown that 
    \begin{align*}
        I^{\check{\alpha}}(x)I^{\check{\alpha}}(p_{G_1}) &= I^{\check{\alpha}}(x p_{G_1}) = I^{\check{\beta}}(x p_{G_1}), \\
        \lambda_D^u(\delta)I^{\check{\alpha}}(p_{G_1}) &= \lambda_D^v(\delta),
    \end{align*}
    where $x \in L(G_1(\omega))$ and $\delta \in D$. Hence, $(L(G_1(\omega)) \rtimes_{(\check{\alpha}, u)}D)I^{\check{\alpha}}(p_{G_1}) = L(G_1(\omega))p_{G_1}\rtimes_{(\check{\beta},v)} D$ holds. Let $\nu(\gamma) \, (\gamma \in \widehat{D})$ be the unitary used to define the action $\hat{\beta}$ dual to $(\check{\beta}, v)$ and $\mu(\gamma) \, (\gamma \in \widehat{D})$ be the one used to define the action $\hat{\alpha}$ dual to $(\check{\alpha}, u)$. Then, we have $\nu(\gamma) = \mu(\gamma)I^{\check{\alpha}}(p_{G_1}) = I^{\check{\alpha}}(p_{G_1})\mu(\gamma)$ for all $\gamma \in \widehat{D}$, and hence it follows that 
    \begin{align*}
        \hat{\beta}_\gamma(a I^{\check{\alpha}}(p_{G_1})) &= \nu(\gamma)a I^{\check{\alpha}}(p_{G_1})\nu(\gamma)^* = \mu(\gamma)I^{\check{\alpha}}(p_{G_1})aI^{\check{\alpha}}(p_{G_1})\mu(\gamma)^* \\ 
        &= \hat{\alpha}_\gamma(aI^{\check{\alpha}}(p_{G_1})) = \hat{\alpha}_\gamma(a)I^{\check{\alpha}}(p_{G_1}),
    \end{align*}
    where $a \in L(G_1(\omega)) \rtimes_{(\check{\alpha}, u)}D$ and $\gamma \in \widehat{D}$. If $\gamma \in \Gamma(\check{\alpha}, u)$, then we have
    $$\hat{\beta}_\gamma(n I^{\check{\alpha}}(p_{G_1})) = \hat{\alpha}_\gamma(n)I^{\check{\alpha}}(p_{G_1}) = nI^{\check{\alpha}}(p_{G_1})$$
    holds for all $n \in Z(L(G_1(\omega)) \rtimes_{(\check{\alpha}, u)}D)$. Note that $Z(L(G_1(\omega))p_{G_1}\rtimes_{(\check{\beta} ,v)} D) = Z((L(G_1(\omega)) \rtimes_{(\check{\alpha}, u)}D)I^{\check{\alpha}}(p_{G_1})) = Z(L(G_1(\omega)) \rtimes_{(\check{\alpha}, u)}D)I^{\check{\alpha}}(p_{G_1})$, we conclude that $\gamma \in \Gamma(\check{\beta}, v)$. By \cite[Proposition 4.4]{Miy25}, we have $\widehat{D} = \Gamma(\check{\alpha},u) \subseteq \Gamma(\check{\beta}, v) = \Gamma(\check{\alpha}^\omega, u^\omega)$. Thus, $\Gamma(\check{\alpha}^\omega, u^\omega) = \widehat{D}$ holds.
\end{proof}

The following results follows by the above result and \cite[Theorem 3.18]{Miy25} (see also \cite[Corollary XI.2.8]{Tak03}).
\begin{cor}\label{cor:factoriality_of_the_twisted_version}
    Let $G$ be a second countable almost unimodular group with a 2-cocycle $\omega:G \times G \to \mathbb{T}$. If we identify $L_\omega(G) \cong L_\omega(\ker\Delta_G) \rtimes_{(\check{\alpha}^\omega,u^\omega)} \Delta_G(G)$ via Theorem~\ref{thm:cocycle_action_decomp}, then the following conditions are equivalent.
    \begin{enumerate}
        \item[(i)] $(\check{\alpha}^\omega, u^\omega): \Delta_G(G) \curvearrowright L_\omega(\ker\Delta_G)$ is centrally ergodic.
        \item[(ii)] $L_\omega(G)$ is a factor.
    \end{enumerate}
\end{cor}

\subsection{Modular spectrum of twisted group von Neumann algebra}
In this subsection, we study the modular spectrum of $\text{S}(L_\omega(G))$ when $L_\omega(G)$ is a factor. First, we recall the following Theorem due to Connes \cite{Con73}. We denote the set of projections of a von Neumann algebra $M$ by $M^P$. 

\begin{thm}[{\!\!\hspace{.1 cm}\cite[Corollary 3.2.5(b)]{Con73}}]\label{thm:Connes_Sinvariant_formula}
    Let $M$ be a factor and $\varphi$ be a normal faithful semifinite weight on $M$. Then, 
        \[
            \S(M)\backslash\{0\}=\bigcap_{0\ne e\in Z(M^{\varphi})^P}\Sp(\Delta_{\varphi_e})\backslash\{0\}
        \]
    where $\varphi_e$ is a weight on $M_e=eMe$ defined by $\varphi_e(x)=\varphi(x)$ for all $x\in M_e$.
\end{thm}

Also, we use the following result which is shown in \cite{Miy25}. Let $G$ be an almost unimodular group. By \cite[II, Proposition 3.1.7]{Sut80}, we can identify $L(G) \cong L(\ker\Delta_G) \rtimes_{(\check{\alpha}, u)} \Delta_G(G)$, where $(\check{\alpha}, u)$ is associated with an arbitrary section $\sigma: \Delta_G(G) \to G$. We call a section $\sigma$ inverse preserving when $\sigma(\delta^{-1})=\sigma(\delta)^{-1}$ holds for all $\delta \in \Delta_G(G)$. In this case, inverse preserving sections always exists (see \cite[Remark 5.2]{Miy25}).

\begin{lem}[{\!\!\hspace{.1 cm}\cite[Lemma 5.4]{Miy25}}]\label{lem:spectrum_formula}
     Let $G$ be a second countable almost unimodular group. We identify $L(G) \cong L(\ker\Delta_G)\rtimes_{(\check{\alpha}, u)} \Delta_G(G)$, where $(\check{\alpha}, u)$ is associated with an inverse preserving section $\sigma: \Delta_G(G) \to G$. Then, for each nonzero $e\in Z(L(G)^{\varphi_G})^P=Z(L(\ker\Delta_G))^P$, 
        \[
            \Sp(\Delta_{(\varphi_G)_e})\backslash\{0\}=\overline{\{\delta\in\Delta_G(G):\check{\alpha}_\delta(e)e\ne0\}}\backslash\{0\}.
        \]
\end{lem}

We provide a twisted version of the above lemma. As it can be seen from the proof of Theorem \ref{thm:cocycle_action_decomp}, the decomposition $L_\omega(G) \cong L_\omega(\ker\Delta_G) \rtimes_{(\check{\alpha}^\omega, u^\omega)} \Delta_G(G)$ is associated with a section $\sigma: \Delta_G(G) \to G$. Moreover, if $\sigma: \Delta_G(G) \to G$ is inverse preserving, then the section $\overline{\sigma}: \Delta_G(G) \to G(\omega)$ defined by $\overline{\sigma}(\delta) := (1, \sigma(\delta))$ is also inverse preserving.

\begin{lem}\label{lem:spectrum_of_modular_operator}
    Let $G$ be a second countable almost unimodular group with a 2-cocycle $\omega:G \times G \to \mathbb{T}$. We identify $L_\omega(G) \cong L_\omega(\ker\Delta_G) \rtimes_{(\check{\alpha}^\omega, u^\omega)} \Delta_G(G)$, where $(\check{\alpha}^\omega, u^\omega)$ is associated with an inverse preserving section $\sigma: \Delta_G(G) \to G$. Then, for each nonzero $e\in Z(L_\omega(G)^{\varphi_G^\omega})^P$, we have 
        \[
            \Sp(\Delta_{(\varphi_G^\omega)_e})\backslash\{0\}=\overline{\{\delta\in\Delta_G(G):\check{\alpha}^\omega_\delta(e)e\ne0\}}\backslash\{0\}.
        \]
\end{lem}
\begin{proof}
    Assume that $\omega$ is fully normalized (see discussion before \cite[Lemma 2]{Kle74} and \cite[Proposition 2.4]{Sut80}). Denote $G_1:=\ker\Delta_G$, $G(\omega) := \mathbb{T}\rtimes_{(1,\omega)} G$, and $G_1(\omega) : = \mathbb{T}\rtimes_{(1,\omega)} G_1$.  Fix a nonzero projection $e \in Z(L_\omega(G)^{\varphi_G^\omega})^P$. Let $\Omega_G: L(G(\omega))p_G \to L_\omega(G)$ be the isomorphism obtained by Proposition \ref{prop:decomp_of_central_extension}. We identify $L(G(\omega))^{\varphi_{G(\omega)}} = \{\lambda_{G(\omega)}(a,s) : (a,s) \in G_1(\omega)\}''$ with $L(G_1(\omega))$, and we identify $L_\omega(G)^{\varphi_G^\omega}$ with $L_\omega(G_1)$ (see \cite[Proposition 2.1, Proposition 2.6]{GG25}). With these identifications, $\Omega_G$ and $\Omega_{G_1}$ agree on 
    $$(L(G(\omega))p_G)^{(\varphi_{G(\omega)})_{p_G}} \cong L(G_1(\omega))p_{G_1}.$$ 
    We know that $\varphi_G^\omega = (\varphi_{G(\omega)})_{p_G} \circ \Omega_G^{-1}$ and $p_G \in Z(L(G(\omega))) \subseteq Z(L(G(\omega))^{\varphi_{G(\omega)}}) \cong Z(L(G_1(\omega)))$. Therefore, it follows that $\Omega_G^{-1}(e) \in Z((L(G(\omega))p_G)^{(\varphi_{G(\omega)})_{p_G}}) = Z(L(G(\omega))^{\varphi_{G(\omega)}}p_G) = Z(L(G(\omega))^{\varphi_{G(\omega)}})p_G \cong Z(L(G_1(\omega)))p_{G_1} \subseteq Z(L(G_1(\omega)))$. 
    Moreover, we have $(\varphi_G^\omega)_e = (\varphi_{G(\omega)})_{p_G \Omega_G^{-1}(e)} \circ \Omega_G^{-1}$. Since $\Omega_G$ provides the spatial isomorphism, we see that $\Sp(\Delta_{(\varphi_G^\omega)_e})\backslash\{0\} = \Sp(\Delta_{(\varphi_{G(\omega)})_{p_G \Omega_G^{-1}(e)}})\backslash\{0\}$. 
    Note that $p_G \Omega_G^{-1}(e)$ corresponds to $p_{G_1} \Omega_{G_1}^{-1}(e)$ under the above identification. 
    By Lemma \ref{lem:spectrum_formula}, we have
    \begin{align*}
        \Sp(\Delta_{(\varphi_G^\omega)_e})\backslash\{0\} &= \overline{\{\delta\in\Delta_G(G):\check{\alpha}_\delta(p_{G_1} \Omega_{G_1}^{-1}(e))p_{G_1} \Omega_{G_1}^{-1}(e)\ne0\}}\backslash\{0\} 
        \\ &= \overline{\{\delta\in\Delta_G(G):\Omega_{G_1}^{-1}(\Omega_{G_1}(\check{\alpha}_\delta(\Omega_G^{-1}(e))) e)\ne0\}}\backslash\{0\}
        \\ &= \overline{\{\delta\in\Delta_G(G):\Omega_{G_1}(\check{\alpha}_\delta(\Omega_G^{-1}(e))) e\ne0\}}\backslash\{0\}
        \\ &= \overline{\{\delta\in\Delta_G(G):\check{\alpha}_\delta^\omega(e)e\ne0\}}\backslash\{0\}.
    \end{align*}
    Here, we use the facts that $\check{\alpha}_\delta(p_{G_1}) = p_{G_1}$ for all $\delta \in \Delta_G(G)$ and $\Omega_{G_1}(p_{G_1}) = 1_{L_\omega(G_1)}$.
\end{proof}

\begin{prop}
    With the same notation as in Lemma \ref{lem:spectrum_of_modular_operator}, we have 
    \begin{align*}
         \text{S}(L_\omega(G)) \backslash \{0\} &= \bigcap_{0 \ne e\in Z(L_\omega(\ker\Delta_G))^P} \overline{ \{\delta \in \Delta_G(G) : \check{\alpha}^\omega_\delta(e)e \ne 0 \}} \backslash \{0\}
         \\ & \supseteq \overline{\{ k \in \Delta_G(G) : \check{\alpha}^\omega_\delta|_{Z(L_\omega(\ker\Delta_G))} = \emph{id} \}} \backslash \{0\}.
     \end{align*}
     Moreover, if $\Delta_G(G)$ is discrete in $\R_+$ (or equivalently $\Delta_G(G)$ is singly generated), then  
     $$\text{S}(L_\omega(G)) \backslash \{0\} = \{ \delta \in \Delta_G(G) : \check{\alpha}^\omega_\delta|_{Z(L_\omega(\ker\Delta_G))} = \emph{id} \} \backslash \{0\}.$$
\end{prop}
\begin{proof}
    Theorem \ref{thm:Connes_Sinvariant_formula} and Lemma \ref{lem:spectrum_of_modular_operator} imply the equation $$\text{S}(L_\omega(G)) \backslash \{0\} = \bigcap_{0 \ne e\in Z(L_\omega(\ker\Delta_G))^P} \overline{ \{\delta \in \Delta_G(G) : \check{\alpha}^\omega_\delta(e)e \ne 0 \}} \backslash \{0\}.$$ 
    The proof of $\bigcap_{0\ne e\in Z(L_\omega(\ker\Delta_G))^P}\{\delta \in \Delta_G(G):\check{\alpha}^\omega_\delta(e)e\ne0\}=\{\delta \in \Delta_G(G):\check{\alpha}^\omega_\delta|_{Z(L(\ker\Delta_G))}=\text{id}\}$, which we omit, is similarly to the proof of \cite[Lemma 5.5]{Miy25}.
\end{proof}

\subsection{T-invariant of twisted group von Neumann algebra}

In this subsection, we study the T-invariant $\text{T}(L_\omega(G))$. First, we provide the characterization of the semifiniteness of the twisted group von Neumann algebra $L_\omega(G)$ through the twisted crossed product decomposition (Theorem \ref{thm:cocycle_action_decomp}).

\begin{prop}\label{prop:condition_of_semifiniteness}
    Let $G$ be a second countable almost unimodular group with a 2-cocycle $\omega:G \times G \to \mathbb{T}$. We identify $L_\omega(G) \cong L_\omega(\ker\Delta_G) \rtimes_{(\check{\alpha}^\omega,u^\omega)} \Delta_G(G)$ with a cocycle action $(\check{\alpha}^\omega,u^\omega): \Delta_G(G) \curvearrowright L_\omega(\ker\Delta_G)$ obtained by Theorem \ref{thm:cocycle_action_decomp}. Then, the following conditions are equivalent.
    \begin{enumerate}
        \item [(i)] $L_\omega(G)$ is semifinite. 
        \item [(ii)] There exists a one parameter unitary group $(v(t))_{t\in \R} \subseteq Z(L_\omega(\ker\Delta_G))$ which satisfies $\check{\alpha}^\omega_\delta(v(t))=\delta^{-it}v(t)$ for all $\delta\in \Delta_G(G)$ and $t \in \R$.
    \end{enumerate}
\end{prop}
\begin{proof}
    The proof is essentially same as that of \cite[Proposition 5.10]{Miy25}, but we include it for completeness. We denote $G_1:=\ker\Delta_G$ and $D :=\Delta_G(G)$.
    Let $\varphi_G^\omega$ be the twisted Plancherel weight of $L_\omega(G)$. Note that $L_\omega(G)^{\varphi_G^\omega} = \{\lambda_G^\omega(s) : s \in G_1\}'' \cong L_\omega(G_1)$. Here, we fix a section $\sigma : D \to G$ for the above decomposition arbitrarily. 

    $\text{(i)}\implies \text{(ii)}.$ By the assumption, there exists a one parameter unitary group $(u(t))_{t \in \R} \subseteq L_\omega(G)$ which satisfies $\text{Ad}(u(t)) = \sigma_t^{\varphi_G^\omega} $ for all $t \in \R$. Then, we have 
    $$\sigma_t^{\varphi_G^\omega}(u(s)) = u(t)u(s)u(t)^* = u(t + s - t) = u(s) \qquad s,t \in \R.$$
    Moreover, for $x \in L_\omega(G)^{\varphi_G^\omega}$, we have $u(t)xu(t)^* = \text{Ad}(u(t))(x) = \sigma_t^{\varphi_G^\omega}(x) = x$. Therefore, it follows that $(u(t))_{t \in \R} \subseteq Z(L_\omega(G)^{\varphi_G^\omega})$ and hence there exists a one parameter unitary group $(v(t))_{t \in \R} \subseteq Z(L_\omega(G_1))$ corresponds to $(u(t))_{t \in \R}$. Put $w(t) = I^{\check{\alpha}}(v(t))$ for all $t \in \R$. Let $\Phi : L_\omega(G_1) \rtimes_{(\check{\alpha}^\omega,u^\omega)} D \to L_\omega(G)$ be the isomorphism in Theorem~\ref{thm:cocycle_action_decomp} and $\psi = \varphi_G^\omega \circ \Phi$. It follows that $\sigma_t^{\psi} = \Phi^{-1} \circ \sigma_t^{\varphi_G^\omega} \circ \Phi$ and can be shown that $\sigma_t^{\psi} = \text{Ad}(w(t))$. 
    We calculate, for $t \in \R, $ and $\delta \in D$, 
    $$\sigma_t^{\psi}(\lambda^u_D(\delta)) = \Phi^{-1}(\sigma_t^{\varphi_G^\omega}(\lambda_G^\omega(\sigma(\delta^{-1})^{-1}))) = \Delta_G(\sigma(\delta^{-1})^{-1})^{it}\lambda^u_D(\delta) = \delta^{it}\lambda^u_D(\delta).$$
    Therefore, we see that $w(t)\lambda_D^u(\delta)w(t)^* = \delta^{it}\lambda_D^u(\delta)$, and hence, by the covariant relation, we have $$I^{\check{\alpha}}(\check{\alpha}_\delta(v(t)))w(t)^* = \lambda_D^u(\delta)I^{\check{\alpha}}(v(t))\lambda_D^u(\delta)^*w(t)^* = \lambda_D^u(\delta)w(t)\lambda_D^u(\delta)^*w(t)^* = \delta^{-it}.$$
    We conclude that $I^{\check{\alpha}}(\check{\alpha}_\delta(v(t))) = \delta^{-it}w(t) = \delta^{-it}I^{\check{\alpha}}(v(t))$, thus $\check{\alpha}_\delta(v(t)) = \delta^{-it}v(t)$ holds for all $t \in \R,$ and $ \delta \in D$.
    
    $\text{(ii)}\implies \text{(i)}.$ Put $w(t) = I^{\check{\alpha}}(v(t))$. Since $(v(t))_{t \in \R} \subseteq Z(L_\omega(G_1))$, it can be shown that 
    \begin{align*}
        \begin{cases}
            \text{Ad}(w(t))(I^{\check{\alpha}}(x)) = I^{\check{\alpha}}(x) = \sigma_t^{\psi}(x), \\
            \text{Ad}(w(t))(\lambda_D^u(\delta)) = \delta^{it}\lambda_D^u(\delta) = \sigma_t^\psi(\lambda_D^u(\delta)),
        \end{cases}
    \end{align*}
    where $t \in \R, x \in L_\omega(G_1), \delta \in D$. Therefore, we conclude that $\sigma_t^{\psi} = \text{Ad}(w(t))$ and hence $L_\omega(G) \cong L_\omega(G_1) \rtimes_{(\check{\alpha}^\omega,u^\omega)} D$ is semifinite.
\end{proof}

Moreover, we also characterize the T-invariant of $L_\omega(G)$ and provide the sufficient condition that the twisted group von Neumann algebra becomes a type $\rm{III}_0$-factor. Since the proof is essentially identical to Proposition \ref{prop:condition_of_semifiniteness} and \cite[Proposition 5.12]{Miy25}, we leave it to the reader. 

\begin{prop}\label{prop:the_T_invariant_of_twisted_version}
     Let $G$ be a second countable almost unimodular group with a 2-cocycle $\omega:G \times G \to \mathbb{T}$. We identify $L_\omega(G) \cong L_\omega(\ker\Delta_G) \rtimes_{(\check{\alpha}^\omega,u^\omega)} \Delta_G(G)$ with a cocycle action $(\check{\alpha}^\omega,u^\omega): \Delta_G(G) \curvearrowright L_\omega(\ker\Delta_G)$ obtained by Theorem \ref{thm:cocycle_action_decomp}. Then, the following conditions are equivalent.
    \begin{enumerate}
        \item [(i)] $t \in \text{T}(L_\omega(G))$.
        \item [(ii)] There exists a unitary $v \in Z(L_\omega(\ker\Delta_G))$ which satisfies $\check{\alpha}_\delta(v)=\delta^{-it}v$ for all $\delta\in \Delta_G(G).$
    \end{enumerate}
\end{prop}

\begin{prop}
    Let $G$ be a second countable almost unimodular group with 2-cocycle $\omega: G \times G \to \mathbb{T}$. Suppose that $\Delta_G(G)$ is relatively discrete and $L_\omega(\ker\Delta_G)$ is not a factor. We identify $L_\omega(G) \cong L_\omega(\ker\Delta_G) \rtimes_{(\check{\alpha}^\omega,u^\omega)} \Delta_G(G)$ through Theorem \ref{thm:cocycle_action_decomp} associated with an inverse preserving section $\sigma : \Delta_G(G) \to G$. If $\check{\alpha}_\delta|_{Z(L_\omega(\ker\Delta_G))}$ is ergodic for each $\delta \in \Delta_G(G) \backslash \{1\}$, then $L_\omega(G)$ is a $\rm{III}_0$-factor. 
\end{prop}

There are examples which can be applied to the above propositions. The following lemma is required for the examples.

\begin{lem}\label{lem:tensor_decomposition_of_product_group}
    Let $G_1, G_2$ be second countable almost unimodular groups with 2-cocycles $\omega_1 : G_1 \times G_1 \to \mathbb{T}$ and $\omega_2 : G_2 \times G_2 \to \mathbb{T}$. It follows that $L_\omega(G_1 \times G_2) \cong L_{\omega_1}(G_1) \bar\otimes L_{\omega_2}(G_2)$ where a 2-cocycle $\omega$ on $G_1 \times G_2$ is defined by 
    $$\omega((s_1, s_2), (t_1, t_2)) = \omega_1(s_1, t_1)\omega_2(s_2, t_2) \qquad s_1, t_1 \in G_1, s_2, t_2 \in G_2.$$
\end{lem}
\begin{proof}
    By calculation, $\omega$ is a 2-cocycle on $G_1 \times G_2$. We will show that $L_\omega(G_1 \times G_2) \cong L_{\omega_1}(G_1) \bar\otimes L_{\omega_2}(G_2)$. Now, let $U$ be the canonical unitary from $L^2(G_1) \otimes L^2(G_2)$ to $L^2(G_1 \times G_2)$. That is, $U : L^2(G_1) \otimes L^2(G_2) \to L^2(G_1 \times G_2)$ is a unitary which satisfies $$U(f_1 \otimes f_2)(s_1, s_2) = f_1(s_1)f_2(s_2) \qquad f_1 \in L^2(G_1), f_2 \in L^2(G_2).$$
    Then, we have 
    \begin{align*}
        [( \lambda_{G_1 \times G_2}^\omega(s_1, s_2) U(f_1 \otimes f_2))](t_1, t_2) &= \omega((t_1, t_2)^{-1}, (s_1, s_2))[U(f_1 \otimes f_2)](s_1^{-1}t_1, s_2^{-1}t_2) \\
        &= \omega_1(t_1^{-1}, s_1)\omega_2(t_2^{-1}, s_2)f_1(s_1^{-1}t_1)f_2(s_2^{-1}t_2) \\
        &= (\lambda_{G_1}^{\omega_1}(s_1)f_1)(t_1)(\lambda_{G_2}^{\omega_2}(s_2)f_2)(t_2) \\
        &= [U(\lambda_{G_1}^{\omega_1}(s_1)f_1 \otimes \lambda_{G_2}^{\omega_2}(s_2)f_2)](t_1, t_2) \\
        &= [U(\lambda_{G_1}^{\omega_1}(s_1) \otimes \lambda_{G_2}^{\omega_2}(s_2))(f_1 \otimes f_2)](t_1, t_2),
    \end{align*}
    where $f_1 \in L^2(G_1), f_2 \in L^2(G_2), s_1, t_1 \in G_1, s_2, t_2 \in G_2$. Linearity and density implies that $\lambda_{G_1 \times G_2}^\omega(s_1, s_2) U = U (\lambda_{G_1}^{\omega_1}(s_1) \otimes \lambda_{G_2}^{\omega_2}(s_2))$ for all $s_1 \in G_1, s_2 \in G_2$. Therefore, we see that $U (\lambda_{G_1}^{\omega_1}(s_1) \otimes \lambda_{G_2}^{\omega_2}(s_2)) U^* = \lambda_{G_1 \times G_2}^\omega(s_1, s_2)$. Since $\lambda_{G_1}^{\omega_1}(s_1),  \lambda_{G_2}^{\omega_2}(s_2)$ and $ \lambda_{G_1 \times G_2}^\omega(s_1, s_2)$ respectively generates $L_{\omega_1}(G_1), L_{\omega_2}(G_2)$ and $L_{\omega}(G_1 \times G_2)$, we conclude that $\text{Ad}(U) : L_{\omega_1}(G_1) \bar\otimes L_{\omega_2}(G_2) \to 
    L_\omega(G_1 \times G_2)$ gives a desired isomorphism.
\end{proof}

\begin{ex}\label{ex:construction_of_III0_factors}
    Let $H$ be a second countable almost unimodular group with a 2-cocycle $\omega : H \times H \to \mathbb{T}$ and set $H_1 := \ker\Delta_H$, where $\Delta_H$ is the modular function of $H$. Suppose that $L_\omega(H_1)$ is a factor and $H$ is not unimodular. Denote $D := \Delta_H(H)$ regarded as a countable discrete abelian group. Set $K := \{0, 1\}^D$ and let $s : D \curvearrowright K$ be the Bernoulli shift, that is, $(s_\gamma (k))_\delta = k_{\gamma^{-1}\delta}$ for all $\gamma, \delta \in D, k \in K$. Also, there is an action $\hat{s} : D \curvearrowright \widehat{K}$ defined by $$\big(k \mid \hat{s}_\gamma(p) \big) = \big(s_\gamma(k) \mid p \big) \qquad k \in K, p \in \widehat{K}, \gamma \in D.$$ 

    Now, we identify $H \cong H_1 \rtimes_{(\alpha^D, c^D)} D$ with an inverse preserving section $\sigma : D \to H$ (See \cite[Proposition 4.1, Example 5.15, Example 5.17]{Miy25} for more details). Then, $G := (H_1 \times \widehat{K}) \rtimes_{(\beta^D, v^D)} D$ is an almost unimodular group where $(\beta^D, v^D) : D \curvearrowright H_1 \times \widehat{K}$ is defined by 
    \begin{align*}
        \begin{cases}
            \beta_\gamma^D(h, p) =(\alpha_\gamma^{D}(h), \hat{s}_\gamma(p)), \quad \gamma \in D, h\in H_1, p \in \widehat{K},\\
            v^D(\gamma, \delta) = (c^{D}(\gamma, \delta), e_{\widehat{K}}), \quad \gamma, \delta \in D.
        \end{cases}
    \end{align*}
    We construct a new 2-cocycle $\tilde{\omega}$ on $G$. For $(h, p, \gamma), (h', p', \gamma') \in G$, define
    $$\tilde{\omega}((h, p, \gamma), (h', p', \gamma')) := \omega(h\sigma(\gamma), h'\sigma(\gamma')).$$ 
    The calculation shows that $\tilde{\omega} : G \times G \to \mathbb{T}$ is a 2-cocycle on $G$. 
    Indeed, for $(h, p, \gamma), (h', p', \gamma'), (h'', p'', \gamma'') \in G$, we have 
    \begin{align*}
        \tilde{\omega}((h, p, \gamma), &(h', p', \gamma')) \tilde{\omega}((h, p, \gamma)(h', p', \gamma'), (h'', p'', \gamma'')) \\
        &= \tilde{\omega}((h, p, \gamma), (h', p', \gamma'))\tilde{\omega}((h\alpha^D_\gamma(h')c^D(\gamma, \gamma'), p\hat{s}_\gamma(p'), \gamma\gamma'), (h'',p'', \gamma'')) 
        \\ &= \omega(h\sigma(\gamma), h'\sigma(\gamma'))\omega(h\alpha^D_\gamma(h')c^D(\gamma, \gamma')\sigma(\gamma\gamma'), h''\sigma(\gamma'')) 
        \\ &= \omega(h\sigma(\gamma), h'\sigma(\gamma'))\omega(h\sigma(\gamma)h'\sigma(\gamma)^{-1}\sigma(\gamma)\sigma(\gamma')\sigma(\gamma\gamma')^{-1}\sigma(\gamma\gamma'), h''\sigma(\gamma'')) 
        \\ &=\omega(h\sigma(\gamma), h'\sigma(\gamma'))\omega(h\sigma(\gamma)h'\sigma(\gamma'), h''\sigma(\gamma''))
        \\ &= \omega(h'\sigma(\gamma'), h''\sigma(\gamma'')) \omega(h\sigma(\gamma), h'\sigma(\gamma')h''\sigma(\gamma''))
        \\ &=\tilde{\omega}((h', p', \gamma'), (h'', p'', \gamma'')) \tilde{\omega}((h, p, \gamma), (h', p' , \gamma')(h'', p'', \gamma'')).
    \end{align*}
    In this construction, we know that $\Delta_G(h, p, \gamma) = \gamma$ and $\Delta_G(G) = D$, and hence $G_1 = (H_1 \times \widehat{K}) \rtimes \{1\} \cong H_1 \times \widehat{K}$ holds. We can construct an inverse preserving section $\sigma^G \colon D \to G$ defined by $\sigma^G(\gamma)=(e_{H_1},e_{\widehat{K}},\gamma)$ for all $\gamma\in D$. With this section, we see that $\sigma^G(\gamma)(h,p,1)\sigma^G(\gamma)^{-1}=(\alpha_\gamma^{D}(h),\hat{s}_\gamma(p),1)$ for each $h\in H_1, p \in \widehat{K}, \gamma\in D$. 
    
    Next, we identify $L_{\tilde{\omega}}(G) \cong L_{\tilde{\omega}}(G_1) \rtimes_{(\check{\alpha}^{\tilde{\omega}}, u^{\tilde{\omega}})} D$ through Theorem \ref{thm:cocycle_action_decomp} associated with the above inverse preserving section $\sigma^G$. With the identification $G_1 = (H_1 \times \widehat{K}) \rtimes \{1\} \cong H_1 \times \widehat{K}$, $\tilde{\omega}|_{G_1 \times G_1}$ corresponds to the 2-cocycle $\omega'$ on $H_1 \times \widehat{K}$ defined by $\omega'((h, p), (h', p')) = \omega(h, h')$. Thus, Lemma~\ref{lem:tensor_decomposition_of_product_group} implies that $L_{\tilde{\omega}}(G_1) \cong L_\omega(H_1) \bar\otimes L(\widehat{K}) \cong L_\omega(H_1) \bar\otimes L^\infty(K)$ and hence $Z(L_{\tilde{\omega}}(G_1)) \cong Z(L_\omega(G_1)) \bar\otimes Z(L^\infty(K)) \cong L^\infty(K)$ by the assumption that $L_\omega(H_1)$ is a factor. Since $$\check{\alpha}_\gamma (\lambda_{G_1(\tilde{\omega})}(1,(h, p, 1))) = \lambda_{G_1(\tilde{\omega})}((1, \sigma^G(\gamma))(1,(h, p, 1))(1, \sigma^G(\gamma)^{-1})) =\lambda_{G_1(\tilde{\omega})}(1,(\alpha_\gamma^{D}(h),\hat{s}_\gamma(p),1))$$
    for each $h\in H_1, p \in \widehat{K}, \gamma\in D$, it can be shown that $(\check{\alpha}^{\tilde{\omega}}, u^{\tilde{\omega}}) : D \curvearrowright Z(L_{\tilde{\omega}}(G_1))$ corresponds to $\bar s : D \curvearrowright L^\infty(K)$ defined by $(\bar s_\gamma (f))(k) = f(s_\gamma(k))$ for all $\gamma \in D, k \in K, f \in L^\infty(K)$. Therefore, we can apply the arguments in \cite[Example 5.15, 5.17]{Miy25} and then conclude that $L_{\tilde{\omega}}(G)$ is a factor of  type
    \begin{align*}
         \begin{cases}
        \rm{III}_0 \quad \text{when}\; \Delta_H(H) \; \text{is relatively discrete,} \\
        \rm{III}_1 \quad \text{when}\; \Delta_H(H) \; \text{is dense in} \; \R_{\ge 0}. \hfill\blacksquare
    \end{cases}
    \end{align*} 
\end{ex}

\begin{rem}
    The group $H$ which satisfies the assumption of Example \ref{ex:construction_of_III0_factors} exists.
    Let $G$ be a second countable unimodular group with a 2-cocycle $\omega_0 : G \times G \to \mathbb{T}$ and $\tilde{G}$ be a second countable non-unimodular almost unimodular group. Suppose that $L_{\omega_0}(G)$ and $L(\text{ker}\Delta_{\tilde{G}})$ are factors. Then, $H = G \times \tilde{G}$ with the 2-cocycle $\omega : H \times H \to \mathbb{T}$ defined by 
    $$\omega((s_1, t_1), (s_2, t_2)) = \omega_0(s_1, s_2) \qquad s_1, s_2 \in G, t_1, t_2 \in \tilde{G}$$
    do the work. Indeed, we see that $\Delta_H(H) = \Delta_{\tilde{G}}(\tilde{G})$ and $L_\omega(H_1) \cong L_{\omega_0}(G) \bar\otimes L(\text{ker}\Delta_{\tilde{G}})$ is a factor by Lemma \ref{lem:tensor_decomposition_of_product_group}. 
    We also know that such a group $G$ exists, for example, any ICC group equipped with a 2-cocycle is desired one.
\end{rem}

We also consider \cite[Example 2.3]{GG25} from the perspective of the twisted crossed product decomposition. 

\begin{ex}\label{ex:analysis_of_GG25_Example2.3}
    Let $G$ be a second countable almost unimodular group. Denote $G_1 := \ker\Delta_G, D:= \Delta_G(G)$. Suppose that $L(G_1)$ is a factor and $G$ is not unimodular. Put $\mathcal{G} := \widehat{D} \times G$. We define the 2-cocycle on $\mathcal{G}$ as follows:
    $$\omega((\gamma_1, s_1), (\gamma_2, s_2)) = \overline{(\gamma_2\mid\Delta_G(s_1))} \qquad \gamma_1, \gamma_2 \in \widehat{D}, s_1, s_2 \in G.$$
    \cite[Example 2.3]{GG25} shows that $L_\omega(\mathcal{G})$ is a semifinite factor and $L(\mathcal{G})$ is purely infinite non-factor. We also analyze this example in another way. 

    Abusing notation slightly, we redefine $\omega$ as follows:
    $$\omega((\gamma_1, s_1), (\gamma_2, s_2)) = (\gamma_1\mid\Delta_G(s_2))^{1/2}\overline{(\gamma_2\mid\Delta_G(s_1))}^{1/2} \qquad \gamma_1, \gamma_2 \in \widehat{D}, s_1, s_2 \in G.$$
    In fact, this 2-cocycle is fully normalized and cohomologous to original one. Thus, we consider $L_\omega(\mathcal{G})$ instead of the original twisted group von Neumann algebra. Let $\sigma_0 : D \to G$ be an arbitrary inverse preserving section. Then, since $\Delta_{\mathcal{G}}(\mathcal{G}) = D$, $\sigma: D \ni \delta \mapsto (e_{\widehat{D}}, \sigma_0(\delta)) \in \mathcal{G}$ defines an inverse preserving section of $\mathcal{G}$. Since $\mathcal{G}_1 := \ker\Delta_{\mathcal{G}}$ is nothing but $\widehat{D} \times G_1$, $\omega |_{\mathcal{G}_1 \times \mathcal{G}_1} = 1$ and hence $L_\omega(\mathcal{G}_1) = L(\mathcal{G})$ holds. By \cite[Corollary 4.3]{Miy25} and Theorem \ref{thm:cocycle_action_decomp}, we have 
    \begin{align*}
        \begin{cases}
            L_\omega(\mathcal{G}) \cong L_\omega(\mathcal{G}_1) \rtimes_{(\check{\alpha}^\omega, u^\omega)} D = L(\mathcal{G}_1) \rtimes_{(\check{\alpha}^\omega, u^\omega)} D \\
            L(\mathcal{G}) \cong L(\mathcal{G}_1) \rtimes_{(\check{\beta}, v)} D 
        \end{cases}
    \end{align*}
    for some cocycle actions $(\check{\alpha}^\omega, u^\omega) : D \curvearrowright L_\omega(\mathcal{G}_1) = L(\mathcal{G}_1)$ and $(\check{\beta}, v) : D \curvearrowright L(\mathcal{G}_1)$ associated with the section $\sigma$. Thus, we may interpret the difference between $L_\omega(\mathcal{G})$ and $L(\mathcal{G})$ as resulting from the change of the cocycle actions due to the 2-cocycle $\omega$. We will observe how the change occurs. 

    For $L(\mathcal{G})$, we study the cocycle action $(\check{\beta}, v)$. Let $(\beta, c) : D \curvearrowright \mathcal{G}_1$ be the cocycle action defined by $\beta_\delta := \text{Ad}(\sigma(\delta)), c(\delta_1, \delta2) := \sigma(\delta_1)\sigma(\delta_2)\sigma(\delta_1 \delta_2)^{-1}$ where $\delta, \delta_1, \delta_2 \in D, (\gamma, s) \in \mathcal{G}_1$. As in the proof of \cite[Proposition 4.1]{Miy25}, we see that $\mathcal{G} \cong \mathcal{G}_1 \rtimes_{(\beta, c)} D$ and $$\check{\beta}_\delta(\lambda_{\mathcal{G}_1}(\gamma, s)) = \lambda_{\mathcal{G}_1}(\beta_\delta(\gamma, s)) = \lambda_{\mathcal{G}_1}(\gamma, \sigma_0(\delta) s \sigma_0(\delta)^{-1}),$$ where $\delta \in D, \gamma \in \widehat{D}, s \in G_1$.
    Note that $L(\mathcal{G}_1) \cong L(\widehat{D}) \bar\otimes L(G_1) \cong \ell^\infty(D) \bar\otimes L(G_1)$ and $Z(L(\mathcal{G}_1)) \cong \ell^\infty(D)$ by the assumption that $L(G_1)$ is a factor, we conclude that $(\check{\beta}, v) : D \curvearrowright Z(L(\mathcal{G}_1))$ corresponds to the trivial action $\text{id} : D \curvearrowright \ell^\infty(D)$ defined by $\text{id}_\delta(f) = f, \,\,\delta \in D, f \in \ell^\infty(D)$. It is clear that $\text{id} : D \curvearrowright \ell^\infty(D)$ is not ergodic and there is no unitary $f_t \in \ell^\infty(D)$ which satisfies $\text{id}_\delta(f_t) = \delta^{-it} f_t \,\,(\delta \in D)$ for each $t \in \R \backslash \{0\}.$ Thus, by \cite[Theorem 4.6, Proposition 5.11]{Miy25}, we conclude that $L(\mathcal{G})$ is a non-semifinite non-factor.

    For $L_\omega(\mathcal{G})$, we study the cocycle action $(\check{\alpha}^\omega, u^\omega
    )$. With the same notation as in the proof of Theorem \ref{thm:cocycle_action_decomp}, we see that 
    \begin{align*}        \check{\alpha}_\delta^\omega(\lambda_{\mathcal{G}_1}^\omega(\gamma,s)) &= \Omega_{\mathcal{G}_1}(\check{\alpha}_\delta(\lambda_{\mathcal{G}_1(\omega)}(1,(\gamma, s)))) = \Omega_{\mathcal{G}_1}(\lambda_{\mathcal{G}_1(\omega)}(\tilde{\sigma}(\delta)(1,(\gamma, s))\tilde{\sigma}(\delta)^{-1})) \\
    &= \Omega_{\mathcal{G}_1}(\lambda_{\mathcal{G}_1(\omega)}((1,(e_{\widehat{D}}, \sigma_0(\delta)))(1,(\gamma, s))(1, (e_{\widehat{D}}, \sigma_0(\delta)^{-1})))) \\
    &= \Omega_{\mathcal{G}_1}(\lambda_{\mathcal{G}_1(\omega)}((\overline{(\gamma \mid \delta)}^{1/2}, (\gamma, \sigma_0(\delta)s))(1, (e_{\widehat{D}}, \sigma_0(\delta)^{-1})))) \\
    &= \Omega_{\mathcal{G}_1}(\lambda_{\mathcal{G}_1(\omega)}(\overline{(\gamma \mid \delta)}, (\gamma, \sigma_0(\delta)s\sigma_0(\delta)^{-1}))) \\
    &= \overline{(\gamma \mid \delta)} \lambda_{\mathcal{G}_1}^\omega(\gamma, \sigma_0(\delta)s\sigma_0(\delta)^{-1}),
    \end{align*}
    where $\delta \in D, \gamma \in \widehat{D}, s \in G_1$. Note that under the identification $L(\widehat{D}) \cong \ell^\infty(D)$, $\lambda_{\hat{D}}(\gamma)$ corresponds to the function $\overline{(\gamma \mid \cdot\,)}$ and $\lambda_{\mathcal{G}_1}^\omega = \lambda_{\mathcal{G}}$. Since $\overline{(\gamma \mid \delta)(\gamma \mid \cdot \,)} = \overline{(\gamma \mid \delta \,\cdot\,)}$ and $\widehat{D}$ generates $\ell^\infty(D)$, using the isomorphism $L(\mathcal{G}_1) \cong \ell^\infty(D) \bar\otimes L(G_1)$, we conclude that $(\check{\alpha}^\omega, u^\omega
    ) : D \curvearrowright Z(L_\omega(\mathcal{G}_1)) = Z(L(\mathcal{G}_1))$ corresponds to the action $\tilde{\alpha} : D \curvearrowright \ell^\infty(D)$ defined by $$(\tilde{\alpha}_{\delta_1} f)(\delta_2)=f(\delta_1 \delta_2) \qquad \delta_1, \delta_2 \in D, f \in \ell^\infty(D).$$ 
    Assume $f \in \ell^\infty(D)$ satisfies $\tilde{\alpha}_\delta f=f$ for all $\delta \in D$. Then, $f(\delta) = (\tilde{\alpha}_\delta f)(1) = f(1)$ holds for all $\delta \in D$. Thus, we have $f \in \C \cdot1_D$ and conclude that $\tilde{\alpha} : D \curvearrowright \ell^\infty(D)$ is ergodic. Moreover, for each $t \in \R$, define $f_t(\delta) = \delta^{-it} $ for all $\delta \in D$. Then, $(f_t)_{t \in \R} \subseteq \ell^\infty(D)$ defines a one parameter unitary group which satisfies $\tilde{\alpha}_\delta f_t = \delta^{-it} f_t$ for all $\delta \in D, t \in \R.$
    By Corollary~\ref{cor:factoriality_of_the_twisted_version} and Proposition~\ref{prop:condition_of_semifiniteness}, we conclude that $L_\omega(\mathcal{G})$ is a semifinite factor. $\hfill\blacksquare$
\end{ex}

We can also consider Corollary~\ref{cor:forms_of_the_twisted_basic_construction} using the same technique as above.

\begin{ex}
    Let $G$ be a second countable almost unimodular group with a 2-cocycle $\omega : G \times G \to \mathbb{T}$. Denote $G_1 := \ker\Delta_G, D:= \Delta_G(G)$. Suppose that $L_\omega(G_1)$ is a factor and $G$ is not unimodular. Put $\mathcal{G} := \widehat{D} \times G$. We define 2-cocycles on $\mathcal{G}$ as follows: 
    \begin{align*}
        \begin{cases}
            \tilde{\omega}((\gamma_1, s_1), (\gamma_2, s_2)) = (\gamma_1\mid\Delta_G(s_2))^{1/2}\overline{(\gamma_2\mid\Delta_G(s_1))}^{1/2}\omega(s_1, s_2), \\
            (\text{id} \times \omega)((\gamma_1, s_1), (\gamma_2, s_2)) = \omega(s_1, s_2) \qquad \gamma_1, \gamma_2 \in \widehat{D}, s_1, s_2 \in G.
        \end{cases}
    \end{align*}
    It can be shown that the 2-cocycle $\tilde{\omega}$ is fully normalized and cohomologous to the 2-cocycle defined in Corollary \ref{cor:forms_of_the_twisted_basic_construction}.
    Similarly to the argument of Example~\ref{ex:analysis_of_GG25_Example2.3}, let $\sigma_0 : D \to G$ be an arbitrary inverse preserving section. Note that $\Delta_{\mathcal{G}}(\mathcal{G}) = D$ and $\sigma: D \ni \delta \mapsto (e_{\widehat{D}}, \sigma_0(\delta)) \in \mathcal{G}$ defines an inverse preserving section of $\mathcal{G}$. Since $\mathcal{G}_1 := \text{ker}\Delta_{\mathcal{G}}$ is nothing but $\widehat{D} \times G_1$, for each $(\gamma_1, s_1), (\gamma_2, s_2) \in \mathcal{G}_1$, we have 
    $\tilde{\omega}((\gamma_1, s_1), (\gamma_2, s_2)) = \omega(s_1, s_2) = (\text{id} \times \omega)((\gamma_1, s_1), (\gamma_2, s_2)).$ 
    Thus, it follows that $L_{\tilde{\omega}}(\mathcal{G}_1) = L_{(\text{id} \times \omega)}(\mathcal{G}_1) \cong L(\widehat{D}) \bar\otimes L_\omega(G_1) \cong \ell^\infty(D) \bar\otimes L_\omega(G_1)$ by Lemma \ref{lem:tensor_decomposition_of_product_group}. 
    By Theorem \ref{thm:cocycle_action_decomp}, we have 
    \begin{align*}
        \begin{cases}
            L_{\tilde{\omega}}(\mathcal{G}) \cong L_{\tilde{\omega}}(\mathcal{G}_1) \rtimes_{(\check{\alpha}^{\tilde{\omega}}, u^{\tilde{\omega}})} D = L_{(\text{id} \times \omega)}(\mathcal{G}_1) \rtimes_{(\check{\alpha}^{\tilde{\omega}}, u^{\tilde{\omega}})} D \\
            L_{(\text{id} \times \omega)}(\mathcal{G}) \cong L_{(\text{id} \times \omega)}(\mathcal{G}_1) \rtimes_{(\check{\beta}^{(\text{id} \times \omega)}, v^{(\text{id} \times \omega)})} D
        \end{cases}
    \end{align*}
    for cocycle actions $(\check{\alpha}^{\tilde{\omega}}, u^{\tilde{\omega}}) : D \curvearrowright L_{\tilde{\omega}}(\mathcal{G}_1) = L_{(\text{id} \times \omega)}(\mathcal{G}_1)$ and $(\check{\beta}^{(\text{id} \times \omega)}, v^{(\text{id} \times \omega)}) : D \curvearrowright L_{(\text{id} \times \omega)}(\mathcal{G}_1)$ associated with the section $\sigma$. We verify that a similar change occurs in twisted group von Neumann algebras $L_{\tilde{\omega}}(\mathcal{G})$ and $L_{(\text{id} \times \omega)}(\mathcal{G})$ as in Example \ref{ex:analysis_of_GG25_Example2.3}.

    Since $Z(L_{\tilde{\omega}}(\mathcal{G}_1)) = Z(L_{(\text{id} \times \omega)}(\mathcal{G}_1)) \cong Z(L(\widehat{D})) \bar\otimes Z(L_\omega(G_1)) \cong \ell^\infty(D)$, we examine how $\lambda_{\mathcal{G}_1}^{(\text{id} \times \omega)}(\gamma, e_{G})$ for all $\gamma \in \widehat{D}$ changes under the influence of the cocycle actions. With the same computation as in Example~\ref{ex:analysis_of_GG25_Example2.3}, we have 
    \begin{align*}
        \begin{cases}
            \check{\alpha}^{\tilde{\omega}}_\delta(\lambda_{\mathcal{G}_1}^{(\text{id} \times \omega)}(\gamma, e_{G})) = \overline{(\gamma \mid \delta)}\lambda_{\mathcal{G}_1}^{(\text{id} \times \omega)}(\gamma, e_{G}), \\
            \check{\beta}^{(\text{id} \times \omega)}(\lambda_{\mathcal{G}_1}^{(\text{id} \times \omega)}(\gamma, e_{G})) = \lambda_{\mathcal{G}_1}^{(\text{id} \times \omega)}(\gamma, e_{G}) \qquad \delta \in D, \gamma \in \widehat{D}.
        \end{cases}
    \end{align*}
    Following Example \ref{ex:analysis_of_GG25_Example2.3}, we conclude that the cocycle action $(\check{\alpha}^{\tilde{\omega}}, u^{\tilde{\omega}}) : D \curvearrowright Z(L_{\tilde{\omega}}(\mathcal{G}_1)) = Z(L_{(\text{id} \times \omega)}(\mathcal{G}_1))$ corresponds to the action $\tilde{\alpha} : D \curvearrowright \ell^\infty(D)$ and hence $L_{\tilde{\omega}}(\mathcal{G})$ is a semifinite factor. On the other hand, we see that the cocycle action $(\check{\beta}^{(\text{id} \times \omega)}, v^{(\text{id} \times \omega)}) : D \curvearrowright Z(L_{(\text{id} \times \omega)}(\mathcal{G}_1))$ corresponds to the trivial action $\text{id} : D \curvearrowright \ell^\infty(D)$ and hence $L_{(\text{id} \times \omega)}(\mathcal{G})$ is a non-semifinite non-factor. $\hfill\blacksquare$
\end{ex}

\bibliographystyle{amsalpha}
\bibliography{references}

\end{document}